\documentclass[oneside,a4paper]{amsart}

\usepackage{etex}
\usepackage{xcolor}
\usepackage[utf8]{inputenc}
\usepackage[T1]{fontenc}
\usepackage[margin=2.5cm]{geometry}
\makeatletter
\def\alloc@#1#2#3#4#5%
{\ifnum\count1#1<#4
	\allocationnumber\count1#1
	\global\advance\count1#1\@ne
	\global#3#5\allocationnumber
	\wlog{\string#5=\string#2\the\allocationnumber}%
	\else\ifnum#1<6
	\def\etex@dummy@definition{}
	\begingroup \escapechar\m@ne
	\expandafter\alloc@@\expandafter{\string#2}#5%
	\else\errmessage{No room for a new #2}\fi\fi
}
\makeatother
\usepackage[thinlines,thiklines]{easybmat}
\usepackage{amsfonts}

\usepackage[all,cmtip]{xy}
\usepackage{amsmath}
\usepackage{amsthm}
\usepackage{amssymb}
\usepackage{enumitem}
\usepackage{booktabs}
\usepackage{tikz}
\usepackage{tikz-cd}
\usepackage{longtable}
\usetikzlibrary{calc,arrows,backgrounds}
\usepackage{float}
\usepackage{floatflt}

\usepackage[backref=page]{hyperref}
\usepackage{hyperref}
\hypersetup{pdftitle={Birational transformation}}
\hypersetup{pdfauthor={Cong Ding}}
\hypersetup{colorlinks=true, linkcolor=red,anchorcolor=red,citecolor=blue}

\theoremstyle{plain}
\newtheorem*{mainthm}{Main Theorem}
\newtheorem{theorem}{\bf Theorem}[section]
\newtheorem{definition}[theorem]{\bf Definition}
\newtheorem{corollary}[theorem]{\bf Corollary}
\newtheorem{lemma}[theorem]{\bf Lemma}
\newtheorem{proposition}[theorem]{\bf Proposition}
\newtheorem{prop-def}[theorem]{\bf Proposition-Definition}
\newtheorem{example}[theorem]{\bf Example}
\newtheorem{remark}[theorem]{\bf Remark}

\newtheorem*{theorem*}{Theorem}

\begin{document}

\title[Birational transformations]{ Birational transformations on irreducible compact Hermitian symmetric spaces}
\author{Cong Ding}
	\address{Morningside Center of Mathematics, Academy of Mathematics \& Systems Science, The Chinese Academy of Sciences, Beijing, China}
	\email{congding@amss.ac.cn}
	\date{\today}
	\maketitle
	
	\begin{abstract} 
		We construct a sequence of explicit blow-ups and blow-downs on irreducible compact Hermitian symmetric spaces $X$ which transforms it into a projective space of the same dimension. Moreover this resolves a birational map given by Landsberg and Manivel. Centers of the blow-ups for $X$ are constructed by loci of chains of minimal rational curves and centers of the blow-ups for the projective space are constructed from the variety of minimal rational tangents of $X$ and its higher secant varieties.
	\end{abstract}
	
\tableofcontents
\section{Introduction}
 It is known that an irreducible compact Hermitian symmetric space is a rational variety, i.e. it is birational to a projecitive space. The goal of this article is to give explicit birational transformations (blow-ups and blow-downs) which transform it into a projective space. Moreover the birational transformations can be uniformly described independent of the classification. In \cite[Section 2.1]{MR1890196}, a rational map $\varphi$ from $\mathbb{P}^n$ to $\mathbb{P}^N$ is written explicitly which gives a birational map between $\mathbb{P}^n$ and an $n$-dimensional irreducible compact Hermitian symmetric space $X$. Assume that the rank of $X$ is $r$, we will find that this birational map can be factorized into 
successive blow-ups for $r-1$ times and successive blow-downs for $r-1$ times, i.e. $\text{Bl}_{r-1}X\cong_{\tilde{\varphi}} \text{Bl}_{r-1}\mathbb{P}^n$ where $\text{Bl}_{r-1}X\rightarrow X, \text{Bl}_{r-1}\mathbb{P}^n\rightarrow \mathbb{P}^n$ are the corresponding successive blow-ups on $X$ and $\mathbb{P}^n$ respectively, $\tilde{\varphi}$ is a resolution of $\varphi$. 

Centers of the blow-ups of $X$ are constructed from chains of minimal rational curves starting from the sink of Bia{\l}ynicki-Birula decomposition (see \cite{MR366940}) of $X$ with respect to the $\mathbb{C}^*$-action on $X$ (see details in Section \ref{construction}). On the other hand, for the centers of the blow-ups of $\mathbb{P}^n$, we need to introduce a subvariety in $\mathbb{P}T_o(X) \cong \mathbb{P}^{n-1}\subset \mathbb{P}^n$ which is called the variety of minimal rational tangents (VMRT for short) of $X$ where $o$ is a base point on $X$. General theory for VMRT on uniruld projective manifolds were developed by Hwang-Mok in a series of works, see for example \cite{MR1748609},\cite{MR2521656} for some surveys. For our own use in the case of irreducible compact Hermitian symmetric spaces, we know there is a minimal equivariant embedding $X\hookrightarrow \mathbb{P}(\Gamma(X,\mathcal{O}(1))^*)$ where $\mathcal{O}(1)$ is the generator of $Pic(X)$. Minimal rational curves are free rational curves of minimal degree with respect to $\mathcal{O}(1)$ which are projective lines through the minimal equivariant embedding.

\begin{definition}\label{VMRTdef}
	The VMRT of $X$ denoted by $\mathcal{C}_o^1(X)\subset \mathbb{P}T_o(X)\cong \mathbb{P}^{n-1}$ is the collection of tangents of free minimal rational curves in $X$ passing through $o$.
\end{definition}

 If we write $X=G/P$ where $G$ is a connected complex simple Lie group and $P$ is a maximal parabolic subgroup, $\mathcal{C}_o^1(X)$ is precisely the highest weight orbit of the isotropy action. Let \[\mathcal{C}_o^k(X)=Sec^{k-1}(\mathcal{C}_o^1(X)):=\overline
{\bigcup_{a_1,\cdots, a_k\in \mathcal{C}_o^1(X)}\mathbb{P}_{a_1a_2\cdots a_k}}\] be the $(k-1)$-th secant variety which is the Zariski closure of union of projective linear subspace spanned by $a_1,a_2,\cdots ,a_k\in \mathcal{C}_o^1(X)$ (generically isomorphic to $\mathbb{P}^{k-1}$).
Then centers of the blow-ups of $\mathbb{P}^n$ are just $\mathcal{C}_o^1(X)$ and strict transforms of $\mathcal{C}_o^k(X)(k\geq 2)$.

Now we can write down the rational map given by Landsberg-Manivel \cite[Section 2.1]{MR1890196} in terms of coordinates. $\varphi: \mathbb{P}^n\dashrightarrow \mathbb{P}^N$ is given as follows
\[\begin{aligned}
\varphi([x_0,x_1,\cdots, x_n])&=[x^r_0,x_0^{r-1}x_1,\cdots,x_0^{r-1}x_n,x^{r-2}_0\mathcal{I}_2(\mathcal{C}_o^1(X)),\cdots x_0^{r-k-1}\mathcal{I}_{k+1}(\mathcal{C}_o^{k}(X)),\cdots, \mathcal{I}_r(\mathcal{C}_o^{r-1}(X))]
\end{aligned}\]
where $\mathcal{I}_{k+1}(\mathcal{C}_o^k(X))$ is the set of generators of the ideal of $\mathcal{C}_o^k(X)\subset \mathbb{P}T_o(X)$ in degree $k+1$.

In \cite{MR1890196} they proved 
\begin{theorem}\label{thmLM}
	The Zariski closure of $\mathrm{Im}(\varphi)$ is isomorphic to the irreducible compact Hermitian symmetric space $X$.
\end{theorem}

We can see that the embedding $X\subset \mathbb{P}^N$ induced by the ratinal map is just the 
minimal equivariant embedding $X\hookrightarrow \mathbb{P}(\Gamma(X,\mathcal{O}(1))^*)$.
The proof of Landsberg-Manivel is based on  local differential geometry.
In this article we give a interpretation of the theorem in terms of loci of chains of minimal rational curves and prove that 

\begin{mainthm}\label{mainthm}
	Let $X$ be an $n$-dimensional irreducible compact Hermitian symmetric space of rank $r\geq 2$. Then $X$ can be transformed into $\mathbb{P}^n$ by successive blow-ups for $r-1$ times and successive  blow-downs for $r-1$ times along smooth centers. More precisely, there exists a sequence of subvarieties $\mathcal{M}_1 \subset \mathcal{M}_2 \subset \cdots \subset \mathcal{M}_{r-1} \subset D$ in $X$ and a sequence of subvarieties $\mathcal{C}_o^1(X) \subset \mathcal{C}_o^2(X) \subset \cdots \subset  \mathcal{C}_o^{r-1}(X) \subset \mathbb{P}T_o(X)\cong \mathbb{P}^{n-1}$ in $\mathbb{P}^n$ such that the following diagram holds, where $D$ is the compactitying divisor with respect to a Harish-Chandra embedding of $\mathbb{C}^n$ in $X$ (i.e. $D$ is the complement of the affine cell in $X$) and $\mathbb{P}^{n-1}$ is embedded as an hyperplane in $\mathbb{P}^n$.
		\[	
		\begin{tikzcd}
		\text{Bl}_{r-1}\mathbb{P}^n\arrow{d}[near start]{
			\tilde{\varphi}}[near end]{\cong}\cdots \arrow{r} &	\text{Bl}_{k}\mathbb{P}^n \cdots \arrow{r}& 
		\text{Bl}_{1}\mathbb{P}^n  \arrow{r}& \mathbb{P}^n\arrow[dashed]{d}{\varphi}\\	
		\text{Bl}_{r-1}X \cdots \arrow{r} &	\text{Bl}_{k}X \cdots \arrow{r}& 
		\text{Bl}_{1}X  \ar{r}& X 	
\end{tikzcd}
	\]
In the diagram $\text{Bl}_{k}\mathbb{P}^n$ and $\text{Bl}_{k}X$ are the $k$-th successive blow-ups of $\mathbb{P}^n$ and $X$ with the smooth centers given by $\tau_{k-1}(\mathcal{C}_o^k(X))$ and $\mu_{k-1}(\mathcal{M}_k)$ respectively where	
$\tau_{k-1}, \mu_{k-1}$ denote the $(k-1)$-th successive strict transforms in corresponding blow-ups. 
	
Moreover, let $E^{X}_k, E^{\mathbb{P}^n}_k$ be the $k$-th exceptional divisors in the blow-ups of $\mathbb{P}^n$ and $X$ respectively, then the order of the exceptional divisors can be reversed in the sense that 
	\[\begin{cases}
		\mu^{(r-k-1)}_{r-1}(E^{X}_k)\cong\tau^{(k-2)}_{r-1}(E_{r-k+1}^{\mathbb{P}^n}), 2\leq k \leq r-1\\
		\mu^{(r-2)}_{r-1}(E^{X}_1)\cong\tau_{r-1}(\mathbb{P}^{n-1}) \\
	    \mu_{r-1}(D)\cong\tau^{(r-2)}_{r-1}(E_{1}^{\mathbb{P}^n})
		 \end{cases}\]
	where we use $\mu^{(r-j-1)}_{r-1}(E^{X}_j), \tau^{(r-j-1)}_{r-1}(E_{j}^{\mathbb{P}^n})$ to denote the strict transforms of $E^{X}_j, E_{j}^{\mathbb{P}^n}$ in the last $(r-j-1)$-times corresponding successive blow-ups. In other words, we can take successive blow-ups on $X$ for $r-1$ times and contract the second to the $(r-1)$-th exceptional divisors in order and finally contract the compactifying divisor we will get the projective space.
\end{mainthm}

	 We will see that $E^{X}_k$ is a $\mathbb{P}^{d_k}$-bundle over $\mu_{k-1}(\mathcal{M}_k)$ where $d_k+1$ is the dimension of the balanced subspace (see Section \ref{tubesub}) with rank $r-k+1$; $E^{\mathbb{P}^n}_k$ is a $\mathbb{P}^{c_k}$-bundle over $\tau_{k-1}(\mathcal{C}
	 _o^k(X))$ where $c_k$ is the dimension of Hermitian characteristic symmetric subspace with $rank=r-k$ (see Section \ref{hermichar}).

Let $Y$ be a subvariety in $X$. The successive strict transform $\mu_{k-1}$ can be given by the following commutative diagram, where the vertical arrows are closed embeddings,  $\sigma_{j}:\text{Bl}_jX\rightarrow \text{Bl}_{j-1}X$ is the $j$-th blow-up morphism. Then the arrows in the first row are just restrictions of $\sigma_1, \cdots, \sigma_k$. Note that $\mu_{k-1}$ does not mean a map, it is used to clarify the notations.
\[	
	\begin{tikzcd}
	\mu_{k-1}(Y)\arrow[hook]{d} \cdots \arrow{r} &	\mu_{j}(Y) \arrow[hook]{d} \cdots \arrow{r}& 
	\mu_1(Y)\arrow[hook]{d}  \arrow{r}& Y\arrow[hook]{d}\\	
	\text{Bl}_{k}X \cdots \arrow{r}{\sigma_{j+1}} &	\text{Bl}_{j}X \cdots \arrow{r}{\sigma_2}& 
	\text{Bl}_{1}X  \arrow{r}{\
	\sigma_1}& X 
	\end{tikzcd}
\] This can be constructed by easy induction where by definition we know $\mu_{j+1}(Y)$ is the Zariski closure of $\sigma^{-1}_{j+1}(\mu_{j}(Y)\backslash\mu_{j}(\mathcal{M}_{j+1}))$ in $\text{Bl}_{j+1}X$. Also $\mu^{(r-k+1)}_{r-1}$, $\tau_{k-1}$ and  $\tau^{(r-k+1)}_{r-1}$ can be given similarly. 

Explicit construction of $\mathcal{M}_k$ will be given in Section \ref{construction}. We will find that $\mathcal{M}_k\backslash \mathcal{M}_{k-1}$ are just the stable manifolds in the Bia{\l}ynicki-Birula decomposition of $X$ with respect to a $\mathbb{C}^*$-action (see Section \ref{BBIHSS} for details). Rank two cases in the Main Theorem can be found in \cite[Chapter III, Theorem 3.8]{MR1234494}, see also \cite{MR3722690}. We now discuss the hyperquadric case as an example.

\begin{example}\label{examhyper}
	The birational map between the projective space $\mathbb{P}^n=\{[x_0,...,x_{n}]\}$ and the hyperquadric $Q^n=\{[z_0,...,z_{n+1}],z_0z_{n+1}=\sum_{i=1}^nz^2_i\}\subset \mathbb{P}^{n+1}$ can be written as follows: \[\varphi([x_0,...,x_n])=[x^2_0,x_0x_1,...,x_0x_n,\sum_{i=1}^nx^2_i]\] where the indeterminacy is $\{x_0=0, \sum_{i=1}^nx^2_i=0\}$ and the image is $\{z_0\neq 0\}\cup \{z_0=z_1=\cdots=z_n=0, z_{n+1}=1\}$. Conversely the inverse rational map can be writte as the projection from the point $[0,...,0,1]$, i.e.
	\[
	\psi([z_0,...,z_{n+1}])=[z_0,...,z_{n}]
	\]
	where the indeterminacy is $\{[0,...,0,1]\}$ and the image is $\{x_0\neq 0\} \cup \{x_0=0,\sum_{i=1}^kx^2_i=0\}$. We have $\text{Bl}_{Q^{n-2}}\mathbb{P}^{n}\cong \text{Bl}_{pt}Q^n$ where $\mathcal{C}_o^1(Q^n)\cong Q^{n-2}\subset \mathbb{P}^{n-1}\subset \mathbb{P}^{n}$ is a hyperquadric and $pt$ means one point. The compactifying divisor $D=\{[0,z_1,\cdots ,z_{n+1}], \sum_{i=1}^nz^2_i=0\}$ is a cone over $Q^{n-2}=\{[0,z_1,\cdots ,z_{n}, 0], \sum_{i=1}^nz^2_i=0\}$ with vertex $[0,...,0,1]$. The exceptional divisor $E^{\mathbb{P}^n}_1$ is a  
	$\mathbb{P}^1$-bundle over $Q^{n-2}$ which is exactly the strict transform of $D$; On the other hand $E^{X}_1\cong \mathbb{P}^{n-1}$ is the strict transform of the hyperplane $\mathbb{P}^{n-1}=\{x_0=0\}\subset \mathbb{P}^n$.
\end{example}

Our main idea is to construct stratifications on $D$ and $\mathbb{P}T_o(X)$ which comes naturally from the rank of $X$ and the well-known polysphere theorem for irreducible compact Hermitian symmetric spaces. 
	\begin{theorem}[The Polysphere Theorem, cf.\cite{MR0404716} or \cite{MR1081948}]
	Suppose that $(X,g)$ is a rank $r$ irreducible Hermitian symmetric space of compact type, $X=G_c/K=G/P$ where $G_c$ is the compact real form of $G$ and $K$ is a maximal compact subgroup of the noncompact real form $G_0$. Then there exists a totally geodesic submanifold $S$ in $X$ isometric to a product of $r$ Riemann spheres equipped with Fubini-Study metric. Moreover $K$-action on $S$ exhausts $X$, i.e. $X=\bigcup_{k\in K}kS$.
\end{theorem}

In fact the Main theorem holds true when $X$ is a polysphere and we will see that this case (although it is not irreducible) actually serves as a model for the problem and the construction for centers is parallel to a general irreducible compact Hermitian symmetric space. The baby case is well-known, where we let $X=\mathbb{P}^1\times \mathbb{P}^1$. We know that if we blow-up $\mathbb{P}^2$ along two distinct points and contract  strict transform of the line connecting these two points we will get $\mathbb{P}^1\times \mathbb{P}^1$. In this case we define the minimal rational curves to be those rational curves of minimal degree with respect to the line bundle $\mathcal{O}(1,1)$ and the VMRT is defined to be two points in the projectivized tangent space (not as the usual definition).

The article is organized as follows. In Section \ref{basics}, we give some preliminaries on Hermitian symmetric spaces.
 In Section \ref{construction}
we give the explicit constructions of the stratifications and the interaction with the Bia{\l}ynicki-Birula decomposition of $X$. In Section \ref{birtrans} we finish the proof of the main theorem.

\subsection*{Notations}In the article we use the following notations. $G(p,q)$ denotes the Grassmannian of $p$-dimensional subspace in a $(p+q)$-dimensional complex vector spaces; $G^{II}(n,n)$ and $G^{III}(n,n)$ denote the orthogonal Grassmannian and Lagrangian Grassmannian respectively;  $Q^n, \mathbb{OP}^2, E_7/P_7$ denote the hyperquadric, the Cayley plane and the Freudenthal variety respectively. 


\section*{Acknowledgement}
I would like to thank Baohua Fu for introducing the problem, and also for some helpful discussions and suggestions. I would also like to thank Hanlong Fang for  drawing my attention on the interation of the construction of centers with the Bia{\l}ynicki-Birula decomposition and Jie Liu for providing the reference \cite{MR2155089}.
I am also grateful to Yifei Chen and Renjie Lyu for some helpful discussions.

\section{Preliminaries in Hermitian symmetric spaces}\label{basics}
We briefly introduce the notion of irreducible compact Hermitian symmetric spaces of tube type, balanced subspaces and characteristic symmetric subspaces in this part. For more information we refer the readers to \cite{MR0404716}, \cite{MR1918134} and \cite{MR1179334}.

\subsection{Hermitian symmetric spaces of tube type} We know an irreducible bounded symmetric domain is said to be of tube type if it is holomorphically equivalent to a tube domain over a self dual cone. Then
an irreducible compact Hermitian symmetric space is called tube type if it is dual to a bounded symmetric domain of tube type. For example, a Lagrangian Grassmannian is dual to a Type III bounded symmetric domain which is biholomorphic to the Siegel upper half plane, so it is of tube type.

Hermitian symmetric spaces of tube type can be described in terms of restricted root systems which will be given precisely in Theorem \ref{RRRthm}. Write $X=G_c/K=G/P$ to be an irreducible compact Hermitian symmetric space with rank $=r$, where $G$ is a connected complex simple Lie group, $P$ is a maximal parabolic subgroup, $G_c$ is a compact real form of $G$ and $K$ is a maximal compact subgroup of the noncompact real form $G_0$.
By Harish-Chandra decomposition, $\mathfrak{g}=\mathfrak{k}^{\mathbb{C}}+\mathfrak{m}^++\mathfrak{m}^-$ where $\mathfrak{g}=Lie(G), \mathfrak{k}=Lie(K), \mathfrak{p}=Lie(P)=\mathfrak{k}^{\mathbb{C}}+\mathfrak{m}^-$. Let $\Delta$ denote the root system and $\Delta^+_{M}$ denote the set of positive noncompact roots whose root spaces are contained in $\mathfrak{m}^+$, other positive roots are called positive compact roots (similarly, negative noncompact and compact roots can be defined). Let $\Pi=\{\alpha_1,\cdots,\alpha_r\}$ be the maximal set of strongly orthogonal roots in $\Delta^+_{M}$ starting from the highest root we know
\begin{theorem}[The Restricted Root Theorem, cf. \cite{MR0161943}]\label{RRRthm}
	Let $\mathfrak{h}^{\mathbb{C}}$ be the 
	Cartan subalgebra.
	Let $\rho$ denote the restriction of roots from $\mathfrak{h}^{\mathbb{C}}$ to $\mathfrak{h}^-=\sum_{\alpha\in \Pi}\sqrt{-1}H_{\alpha}\mathbb{R}$, identify the elements in $\Delta$ with their $\rho$-image, then either $\rho(\Delta)\cup\{0\}=\{\pm\frac{1}{2}\alpha_i\pm\frac{1}{2}\alpha_j:1\leq i,j\leq r\}$ or $\rho(\Delta)\cup\{0\}=\{\pm\frac{1}{2}\alpha_i\pm\frac{1}{2}\alpha_j, \pm\frac{1}{2}\alpha_i:1\leq i,j\leq r\}$. Accordingly $\rho(\Delta^+_M)=\{\frac{1}{2}\alpha_i+\frac{1}{2}\alpha_j:1\leq i,j\leq r\}$ or $\rho(\Delta^+_M)=\{\frac{1}{2}\alpha_i+\frac{1}{2}\alpha_j, \frac{1}{2}\alpha_i:1\leq i,j\leq r\}$. Moreover, all $\alpha_i$ have the same length and the subgroup of the Weyl group of $G$ preserving the compact roots and fixing $\Pi$ as a set induces all signed permutations $\alpha_i\rightarrow \pm\alpha_j$ of $\Pi$.
\end{theorem}

Moreover, this gives two cases for the restricted compact and noncompact roots, 
\begin{enumerate}
	\item Tube type, which is corresponding to the first case in the theorem. In that case, nonzero $\rho$-image of some subsets of $\Delta$ are given by 
	\begin{itemize}
		\item Compact simple roots $\{\frac{1}{2}(\alpha_{t+1}-\alpha_{t}), 1\leq t \leq  r-1\}$
		\item Compact positive roots $\{\frac{1}{2}(\alpha_{s}-\alpha_{t}), 1\leq t <s\leq  r\}$
		\item Noncompact positive roots $\{\frac{1}{2}(\alpha_{s}+\alpha_{t}), 1\leq t \leq s\leq  r\}$
	\end{itemize}
	\item Non-tube type, which is corresponding to the second case in the theorem. Then nonzero $\rho$-image of some subsets of $\Delta$ are given by 
	\begin{itemize}
		\item Compact simple roots $\{\frac{1}{2}(\alpha_{t+1}-\alpha_{t}), 1\leq t \leq  r-1\}\cup \{-\frac{1}{2}\alpha_{r}\}$
		\item Compact positive roots $\{\frac{1}{2}(\alpha_{s}-\alpha_{t}), 1\leq t <s\leq  r\}\cup \{-\frac{1}{2}\alpha_{t}, 1\leq t\leq r\}$
		\item Noncompact positive roots $\{\frac{1}{2}(\alpha_{s}+\alpha_{t}), 1\leq t \leq s\leq  r\}\cup \{\frac{1}{2}\alpha_{t}, 1\leq t\leq r\}$
	\end{itemize}
\end{enumerate}
 Hermitian symmetric spaces of tube type with rank $\geq 2$ can be listed as
(1) $G(n,n)$ with $n\geq 2$; (2) $G^{II}(n,n)$ with $n$ even and $n\geq 4$;
(3) $G^{III}(n,n)$ with $n\geq 2$;
(4) $Q^n$ with $n\geq 3$;
(5) $E_7/P_7$.
 Also they can be also characterized by generic tangent vectors (i.e. tangent vectors of maximal rank), where rank of a holomorphic tangent vector is given by the following (see \cite[p.234]{MR1081948}).
\begin{prop-def}\label{HSS-normal}
	For a nonzero holomorphic tangent vector $v\in T_o(X)$, there exists some $k\in K$ such that $Adk.v=\sum_{i=1}^{r_0}a_ie_{\alpha_i}$ with $a_1\geq a_2 \cdots \geq a_{r_0}>0$ for some $1\leq r_0\leq r$. Moreover by $\mathbb{C}^*$-action induced by the Cartan subalgebra $\mathfrak{h}^{\mathbb{C}}$,
	there exists some $p\in P$ such that $Adp.v=\sum_{i=1}^{r_0}e_{\alpha_i}$. Here $r_0$ is called the rank of $v$.
\end{prop-def}

We explain this proposition using the Grassmannian $X=G(p,q)$. In this case, $K=S(U(p)\times U(q))$ and $K^{\mathbb{C}}=S(GL(p,\mathbb{C})\times GL(q, \mathbb{C}))$ is the reductive part of $P$, and $T_o(X)\cong \mathfrak{m}^+$ can be identified with the set of $p\times q$ matrices. The actions of $K$ and $P$ are elementary transformations on matrices and the rank of a tangent vector is precisely the rank of the corresponding matrix.
The VMRT of $X$ is exactly the collection of projectivization of rank one tangent vecters at $o$. Then Hermitian symmetric spaces of tube type can be characterized by the following proposition (see \cite[Proposition 1]{MR1918134}). 

\begin{proposition}\label{tubehyper}
	The set of generic tangent vectors (i.e. tangent vectors of maximal rank) is equal to the complement of a (degree $r$) hypersurface in $\mathbb{P}T_o(X)$ if and only if $X$ is of tube type.
\end{proposition}

\begin{table}[h] 
	
	\caption{VMRTs for irreducible compact Hermitian symmetric spaces}
	\begin{tabular}{|p{9em}|p{8em}|p{10em}|}
		\hline
		$X$  & the VMRT $\mathcal{C}_o^1(X)$   & Embedding in $\mathbb{P}T_o(X)$ \\
		\hline
		$G(p,q)$   &  $\mathbb{P}^{p-1}\times \mathbb{P}^{q-1}$ & Segre embedding\\
		\hline
		$G^{II}(n,n)$ & $G(2,n-2)$ & Pl\"{u}ker embedding\\
			\hline
		$G^{III}(n,n)$ & $\mathbb{P}^n$ & Veronese embedding \\
			\hline
		$Q^n$ & $Q^{n-2}$ & $\mathcal{O}(1)$ \\ 	\hline
		
		$\mathbb{OP}^2$ & $G^{II}(5,5)$ & $\mathcal{O}(1)$\\
			\hline
		$E_7/P_7$ & $\mathbb{OP}^2$ &  Severi \\
		\hline
	\end{tabular}
\end{table}
\subsection{Balanced subspaces}\label{tubesub}
Next we introduce the notion of balanced subspace with rank $k$. They are a class of special Hermitian symmetric subspaces. First we introduce the notion of invariantly geodesic subspace which was discussed in \cite{MR1198602} and can be characterized in terms of Lie triple systems. 

\begin{definition}
	Fix a canonical K\"ahler-Einstein metric $h$ on $X=G/P$. A complex submanifold $S\subset X$ is said to be invariantly geodesic if and only of $g(S)$ is totally geodesic in $(X,h)$ for any $g\in G$.
\end{definition}

Full classification of invariantly geodesic subspaces can be found in \cite{MR1198602}.
These Hermitian symmetric subspaces are actually associted to subdiagrams of the marked Dynkin diagram of $X$ (see Definition \ref{subdiagram}) and vice versa. From \cite[Proposition 3.7]{MR3019452} (see also \cite{MR1703350}), they are smooth Schubert varieties in $X$ and vice versa.
\begin{definition}\label{subdiagram}
	Let $(S, X)$ be a pair of irreducible Hermitian symmetric spaces
	associated to the marked Dynkin diagrams $(\mathcal{D}_0, \gamma_0)$ and $(\mathcal{D}, \gamma)$ respectively, then the pair is called subdiagram type if $\mathcal{D}_0$ can be obtained from a subdiagram of $\mathcal{D}$ by identifying $\gamma_0$ with $\gamma$. Then $S$ is said to be associated to a subdiagram of the marked Dykin diagram of $X$.
\end{definition} 

We give an example of a Hermitian symmetric subspace $Q^8\subset \mathbb{OP}^2$ associated to a subdiagram of $E_6$ here (see Figure \ref{E6figure}, where the black node $\alpha_1$ is marked).
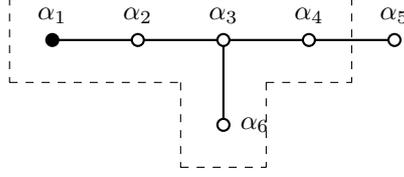
\begin{figure}[h]
	\centering
	\caption{Marked Dyndin diagram of $E_6$ and the subdiagram associated to $Q^8$} \label{E6figure} 
	\begin{tikzpicture}[scale=0.75]
	\draw[dashed] (-0.75,0.75) -- (-0.75, -0.75)
	(-0.75,0.75) -- (5.25, 0.75)   (-0.75,-0.75) -- (2.25, -0.75) 
	(2.25,-0.75) -- (2.25, -2.25) 
	(3.75,-0.75) -- (3.75, -2.25) 
	(2.25, -2.25) -- (3.75, -2.25) 
	(3.75,-0.75) -- (5.25, -0.75) 
	(5.25, 0.75) -- (5.25, -0.75) ;
	\draw[thick] (0,0) -- (6,0)  (3,0) -- (3,-1.5) ;
	
	\draw[ thick, fill=black] (0,0) circle (3pt) node[above, outer sep=3pt]{$\alpha_1$};
	\draw[ thick, fill=white] (1.5,0) circle (3pt) node[above, outer sep=3pt]{$\alpha_2$};
	\draw[ thick, fill=white] (3,0) circle (3pt) node[above, outer sep=3pt]{$\alpha_3$};
	\draw[ thick, fill=white] (4.5,0) circle (3pt) node[above,outer sep=3pt]{$\alpha_4$};
	\draw[ thick, fill=white] (6,0) circle (3pt) node[above,outer sep=3pt]{$\alpha_5$};
	\draw[ thick, fill=white] (3,-1.5) circle (3pt) node[right,outer sep=3pt]{$\alpha_6$};
	\end{tikzpicture} 	
\end{figure}

 Then for our convenience in this article we introduce the following notion.
 \begin{definition}[Balanced subspace]
 We say a complex submanifold $S$ is a balanced subspace if it is an invariantly geodesic subspace and it is of tube type.
 \end{definition}
 
 From \cite{MR1198602} we have the full classification for balanced subspaces, see Table \ref{subspace}. From the perspective of Restricted Root Theorem we can easily observe that 
 
 \begin{proposition}\label{tubetangent}
 Up to $P$-action, the tangent space at $o$ of a balanced subspace with rank $k$ can be identified with 
 \[
 \mathbb{C}\{e_{\beta}, \beta|_{\mathfrak{h}^-}=\frac{1}{2}(\alpha_{\ell}+\alpha_j), (1\leq \ell \leq j\leq k) \}\]
  \end{proposition}

\subsection{Hermitian characteristic symmetric subspaces}\label{hermichar}
We refer the readers to \cite{MR0404716} for this part.
Let $\mathfrak{g}_{\Lambda}$ be the derived algebra of 
\[\mathfrak{h}^{\mathbb{C}}+\sum_{\phi \perp \Pi\backslash \Lambda}\mathfrak{g}_{\phi}\]
 where the orthogonality is with respect to the metric induced by the killing form where $(\phi,\psi)=B(H_{\phi}, H_{\psi})$.
 \begin{definition}[Hermitian characteristic symmetric subspaces]
  Let $G_{\Lambda}$ be the Lie subgroup in $G$ for $\mathfrak{g}_{\Lambda}$ and $X_{\Lambda}$ be the orbit $G_{\Lambda}.o$ on $X$. Then the Hermitian symmetric subspace $X_{\Lambda}$ is called $|\Lambda|$-th characteristic symmetric subspaces.
\end{definition}
Also Hermitian characteristic symmetric subspaces are invariantly geodesic in $X$, they are smooth Schubert varieties and they are associated to subdiagrams of the marked Dynkin diagram of $X$. We can observe that
\begin{proposition}\label{chactertangent}
$T_o(X_\Lambda)$ can be identified with
\[\mathfrak{m}_{\Lambda}^+:=\mathbb{C}\{e_{\beta}, \beta|_{\mathfrak{h}^-}=\frac{1}{2}(\alpha_{\ell}+\alpha_j)\ \text{or} \ \frac{1}{2}\alpha_{\ell},   \alpha_{\ell},\alpha_j\in \Lambda\}\] 
in terms of restricted roots.
\end{proposition}
 
\begin{example}
	We give the balanced subspaces and Hermitian characteristic symmetric subspaces of a Grassmannian $G(p,q)(p\geq q\geq 2)$ here. It will be more straightforward if we consider the tangent space of $X$ at $o$ (or alternatively the Harish-Chandra coordinate on $X$). We know 
	$T_o(X)\cong \mathfrak{m}^+=\{\begin{bmatrix}
		0 & A\\
		0 & 0
		\end{bmatrix}, A\in M(p,q;\mathbb{C})\}$. We identify the elements in $\mathfrak{m}^+$ with $p\times q$ matrices. Let $E_{ij}$ be the $p\times q$ matrix with unique nonzero element $1$ in $i$-th row and $j$-th column. Then (under the choice of Cartan subalgebra and root systems) the root vectors associated to the maximal set $\Pi$ of strongly orthogonal roots are $E_{ii} (1\leq i\leq q)$. The balanced subspace with rank $k$ has tangent space $\mathbb{C}\{E_{ij}, 1\leq i,j\leq k\}$. Correspondingly if $\Lambda$ is the set of roots with root vectors $E_{ii}(k+1\leq i\leq q)$, then the Hermitian characteristic symmetric subspaces with rank $r-k$ has tangent space $\mathbb{C}\{E_{ij}, k+1\leq i\leq p,k+1\leq j\leq q\}$.
\end{example}
For the readers' convenience we list a table for full classification of balanced subspaces and characteristic symmetric subspaces (see Table \ref{subspace}).
\begin{table}[h]
	\centering
	\caption{Balanced subspaces and characteristic symmetric subspaces}\label{subspace}
	\begin{tabular}{|p{8em}|p{12em}|p{12em}|}
		\hline
		$X$ &   Balanced subspace with rank $k$ &  Hermitian characteristic symmetric subspace with rank $r-k$\\  \hline
		$G(p,q) (p\geq q\geq 2)$ & $G(k,k)$   & $G(q-k, p-k)$\\ \hline
		$G^{II}(n,n)(n\geq 4)$ &  $G^{II}(2k,2k)$ & $G^{II}(n-2k,n-2k)$\\
		\hline
		$G^{III}(n,n)(n\geq 2)$ &   $G^{III}(k,k)$ &  $G^{III}(r-k,r-k)$\\ \hline
		$Q^n(n\geq 3)$&  $\mathbb{P}^1, k=1$
		& $\mathbb{P}^1,k=1$\\ \hline
		$\mathbb{OP}^2$ & $\mathbb{P}^1$ when $k=1$, $Q^{8}$ when $k=2$ 
		& $\mathbb{P}^5,k=1$\\
		\hline
		$E_7/P_7$ &$\mathbb{P}^1$ when $k=1$, $Q^{10}$ when $k=2$ &  $Q^{10}$ when $k=1$, $\mathbb{P}^1$ when $k=2$ \\
		\hline
	\end{tabular}
\end{table}

Finally in this section we give the following proposition on smooth Schubert varieties in $X$.

\begin{proposition}\label{subvmrt}
	Let $S\subset X$ be a smooth Schubert variety in $X$ (i.e. $(S,X)$ is of subdiagram type). Then $\mathcal{C}_o^k(S)=\mathcal{C}_o^k(X)\cap \mathbb{P}T_o(S)$ where $\mathcal{C}_o^k(S)$ denotes the $(k-1)$-th secant variety of the VMRT of $S$ at $o$.
\end{proposition}
\begin{proof}
	Since $S$ is a linear section of $X\hookrightarrow \mathbb{P}(\Gamma(X,\mathcal{O}(1))^*)$ and the embedding $S\subset X$ preserves minimal rational curves, we know $\mathcal{C}_o^1(S)=\mathcal{C}_o^1(X)\cap \mathbb{P}T_o(S)$. Then $\mathcal{C}_o^k(S)\subset \mathcal{C}_o^k(X)\cap \mathbb{P}T_o(S)$ for any $1\leq k\leq r_0=rank(S)$. Write $S$ as $S=G'/P'$ where $P'$ can be chosen to be a subgroup of $P$. We know $\mathcal{C}_o^k(S)$ are all $P'$-invariant subvarieties in $\mathbb{P}T_o(S)$. On the other hand $ \mathcal{C}_o^k(X)\cap \mathbb{P}T_o(S)$ is also $P'$-invariant. Hence the inclusion is actually an equality. 
\end{proof}

\section{The Bia{\l}ynicki-Birula decomposition and constructions of the centers}\label{construction}

In this section we give the construction for the centers and describle the relation with the Bia{\l}ynicki-Birula decomposition of irreducible compact Hermitian symmetric spaces.
 
\subsection{The loci of infinity points of degree $k$ rational curves}\label{infinityloci}

For further discussion we first give a lemma which identify the higher secant varieties with the sets of tangent vectors with certain ranks.
\begin{lemma}\label{secantrank}
	$\mathcal{C}_o^k(X)=\{[v]\in \mathbb{P}T_o(X)\cong \mathbb{P}^{n-1}, rank(v)\leq k\}$ where the bracket denotes projectivization of a vector.
\end{lemma}
\begin{proof}
	The lemma is straightforward when $rank(X)=2$, which includes the case where $X$ is a hyperquadric or a Cayley plane.
	From Proposition \ref{HSS-normal} we know $\mathcal{C}_o^k(X)\supset\{[v]\in \mathbb{P}T_o(X)\cong \mathbb{P}^{n-1}, rank(v)\leq k\}$. For the other inclusion, it suffices to show that $rank(v+v')\leq rank(v)+rank(v')$ for $v,v'\in T_o(X)$. This can be checked case by case. In the Grassmannian case $X=G(p,q)$, a tangent vector is identified with a $p$-by-$q$ matrix and rank of the tangent vector is exactly the rank of the matrix. In the orthogonal Grassmannian case $X=G^{II}(n,n)$ a tangent vector is identified with an antisymmetric $n$-by-$n$ matrix and rank of the tangent vector is exactly half of the rank of the matrix. In the Lagrangian Grassmannian case $X=G^{III}(n,n)$ a tangent vector is identified with a symmetric $n$-by-$n$ matrix and rank of the tangent vector is exactly the rank of the matrix. Hence the inequality holds for these three cases by basic linear algebra. When $X=E_7/P_7$, $\mathcal{C}_o^1(E_7/P_7)\cong\mathbb{OP}^2$, and a tangent vector can be identified with an element in the following Jordan algebra
	\[J_3(\mathbb{O})=\{\begin{bmatrix}
	c_1 & x & y \\
	\overline{x} & c_2 & z \\
	\overline{y} & \overline{z} & c_3
	\end{bmatrix}, c_1,c_2,c_3\in \mathbb{C}, x,y,z\in \mathbb{O}\}\]
	and $\mathcal{C}_o^1(E_7/P_7)\cong \mathbb{OP}^2\subset \mathbb{P}T_o(E_7/P_7)$ can be identified with the set of rank one matrices in $J_3(\mathbb{O})$. Inequality for the rank still holds and hence the lemma follows. 
\end{proof}

Let $W\cong \mathbb{C}^n \subset X$ be the Harish-Chandra embedding given by the map $\exp: \mathfrak{m}^+ \rightarrow \exp(\mathfrak{m}^+).o$ where $o$ is the zero point. 
Let \[\begin{aligned}
N_k=&\{C\backslash W: C \ \text{is a rational curve passing through $o$ such that} \\& [T_o(C)] \ \text{is of rank} \ k\  \text{and} \ C\cap W \ \text{is an affine line}\} 
\end{aligned}\] 
i.e. $N_k$ is the collection of infinity points of all such degree $k$ rational curves (which can be written as $\exp(tv).o$ with $[v]\in \mathcal{C}_o^k(X)\backslash\mathcal{C}_o^{k-1}(X)$). We will call $N_k$ the $k$-th infinity locus of $X$.
 
The purpose of this section is to compute $N_k(1\leq k\leq r)$ explicitly. 
We fix the normalization on root vectors as follows, let $H_\alpha$ denote the dual of $\alpha$ with respect to the killing form. Then $h_{\alpha}=\frac{2H_{\alpha}}{<\alpha,\alpha>}$ is the coroot of $\alpha$. $e_{\alpha}$ is normalized so that $[e_{\alpha},e_{-\alpha}]=h_{\alpha}$.
Let $x_{\alpha}=\sqrt{-1}(e_{\alpha}+e_{-\alpha}), y_{\alpha}=-(e_{\alpha}-e_{-\alpha})$

For a minimal rational curve $C$ passing through $o$ with $T_{o}(C)=\mathbb{C}e_{\alpha}$, we have
\begin{lemma}
	The action $\exp(\frac{\pi}{2}x_{\alpha})$ translates $o$ into the infinity point of $C$.
\end{lemma}
\begin{proof}
	The $\mathfrak{sl}_2$-triple $\mathbb{C}\{e_{\alpha},e_{-\alpha},h_{\alpha}\}$ corresponding to the minimal rational curve can be identified with
	$\small{e_{\alpha}=\begin{bmatrix}
		0& 1\\
		0 & 0
		\end{bmatrix}}$,$\small{ e_{-\alpha}=\begin{bmatrix}
		0& 0\\
		1 & 0
		\end{bmatrix}}$, $
	\small{h_{\alpha}=\begin{bmatrix}
		1& 0\\
		0 & -1
		\end{bmatrix}}$. Then we can easily compute \[\small{\exp(\frac{\pi}{2}x_{\alpha})=\exp(\begin{bmatrix}
		0& \frac{\pi\sqrt{-1}}{2}\\
		\frac{\pi\sqrt{-1}}{2} & 0
		\end{bmatrix})=\begin{bmatrix}
		0& \sqrt{-1}\\
		\sqrt{-1} & 0
		\end{bmatrix}},\] which translates $(0,1)^T$ into $(\sqrt{-1},0)^T$. Thus $\exp(\frac{\pi}{2}x_{\alpha})$ translates $o$ into the infinity point of $C$.
\end{proof}

Next we compute the action of $Ad\exp(\frac{\pi}{2}x_{\alpha})$ on $\mathfrak{k}^{\mathbb{C}}$.

\begin{lemma}\label{inftra}
	$Ad\exp(\frac{\pi}{2}x_{\alpha})h_\beta=h_{\beta}-\alpha(h_{\beta})h_{\alpha}$. Moreover $Ad\exp(\frac{\pi}{2}x_{\alpha})$ transforms the $\mathfrak{sl}_2$-triple associated to $\beta$ to the $\mathfrak{sl}_2$-triple associated to $\beta$ or $\alpha+\beta$ or $-\alpha+\beta$, corresponding to the cases when $\beta$ is strongly orthogonal to $\alpha$, or $\alpha+\beta$ is a root or $-\alpha+\beta$ is a root respectively.
\end{lemma}
\begin{proof}
	We can compute
	\begin{align*}
	&Ad\exp(\frac{\pi}{2}x_{\alpha})h_\beta=\exp(ad(\frac{\pi}{2}x_{\alpha}))h_{\beta}\\
	&=h_{\beta}+\sqrt{-1}\alpha(h_{\beta})\sum_{k=0}^{\infty}(\frac{\pi}{2})^{2k+1}\frac{(-4)^k}{(2k+1)!}y_{\alpha}-2\alpha(h_{\beta})\sum_{k=1}^{\infty}(\frac{\pi}{2})^{2k}\frac{(-4)^{k-1}}{(2k)!}h_{\alpha}\\
	&=h_{\beta}+\frac{1}{2}\alpha(h_{\beta})\sum_{k=0}^{\infty}\frac{(\pi\sqrt{-1})^{2k+1}}{(2k+1)!}y_{\alpha}+\frac{1}{2}\alpha(h_{\beta})\sum_{k=1}^{\infty}\frac{(\pi\sqrt{-1})^{2k}}{(2k)!}h_{\alpha}\\
	&=h_{\beta}+\frac{1}{2}\alpha(h_{\beta})\sinh(\pi\sqrt{-1})y_{\alpha}+\frac{1}{2}\alpha(h_{\beta})(\cosh(\pi\sqrt{-1})-1))h_{\alpha}\\
	&=h_{\beta}-\alpha(h_{\beta})h_{\alpha}
	\end{align*}
	Also since $\alpha$ is a root with long length, maximal $\alpha$-string attached to $\beta$ is $\beta, \beta+\alpha$ or $\beta-\alpha, \beta$ (cf. \cite[Appendix I.3, Proposition 3]{MR1081948}), which implies that either $\alpha+\beta$ or  $-\alpha+\beta$ is not a root. When 
	$\beta$ is strongly orthogonal to $\alpha$, $Ad\exp(\frac{\pi}{2}x_{\alpha})e_\beta=e_{\beta}$. When $\alpha+\beta$ is a root, we have $[x_\alpha, [x_\alpha, e_\beta]]=[e_{-\alpha},[e_{\alpha},e_\beta]]=[-h_{\alpha},e_{\beta}]=-\beta(h_{\alpha})e_{\beta}=e_\beta$. From \cite[Appendix I.3, Proposition 3]{MR1081948}) we know $[e_{\alpha}, e_{\beta}]=\sqrt{\frac{<\beta+\alpha, \beta+\alpha>}{<\beta, \beta>}}e_{\alpha+\beta}$. Also we know $\beta(h_{\alpha})=\frac{2<\alpha, \beta>}{<\alpha, \alpha>}=-1$, thus $[e_{\alpha}, e_{\beta}]=e_{\alpha+\beta}$. Then
	\begin{align*}
	&Ad\exp(\frac{\pi}{2}x_{\alpha})e_\beta=\exp(ad(\frac{\pi}{2}x_{\alpha}))e_{\beta}\\
	&=\sum^{\infty}_{k=0}\frac{(\pi\sqrt{-1})^{2k+1}}{2^{2k+1}(2k+1)!}e_{\alpha+\beta}+e_{\beta}+\sum_{k=1}^{\infty}\frac{(\pi\sqrt{-1})^{2k}}{2^{2k}(2k)!}e_{\beta}\\
	&=\sinh(\frac{\pi\sqrt{-1}}{2})e_{\alpha+\beta}+e_{\beta}+(\cosh(\frac{\pi\sqrt{-1}}{2})-1)e_{\beta}\\
	&=\sqrt{-1}e_{\alpha+\beta}
	\end{align*}
	When $-\alpha+\beta$ is a root, the computation is similar.
\end{proof}

We can obtain $N_k$ from the Restricted Root Theorem. 
The complexification of $K$, denoted by $K^\mathbb{C}$, is a reductive part of $P$. Then we know the $N_k$ is a $K^\mathbb{C}$-orbit and moreover $N_k=K^\mathbb{C}c_{\Gamma}o$ where $\Gamma\subset \Pi$ with $|\Gamma|=k$ and  $c_{\Gamma}=\exp(\frac{\pi}{2}\sum_{\alpha_i\in \Gamma}x_{\alpha_i})$. As $Adc^2_{\Gamma}=id$ we have $N_k= c_{\Gamma}c_{\Gamma}^{-1}K^\mathbb{C}c_{\Gamma}o=c_{\Gamma}(Adc_{\Gamma}K^\mathbb{C}).o$.
Let $\mathfrak{k}^{\mathbb{C}}=Lie(K^\mathbb{C})$. To obtain $N_k$, it suffices to know $Adc_{\Gamma}\mathfrak{k}^\mathbb{C}\cap\mathfrak{m}^+$, this can be done by the Restricted Root theorem together with Lemma \ref{inftra}. We have
\begin{proposition}\label{inftypart}
	$Adc_{\Gamma}\mathfrak{k}^\mathbb{C}\cap\mathfrak{m}^+=\mathbb{C}\{e_{\beta}, \beta\in \Theta\} $ where 
	$\beta\in \Theta\subset \Delta^+_{M}$ if there exists only one root $\alpha_s\in \Gamma$ such that $\alpha_s-\beta$ is a root. In particuler, when $X$ is of tube type, $Adc_{\Gamma}\mathfrak{k}^\mathbb{C}\cap\mathfrak{m}^+=Adc_{\Pi-\Gamma}\mathfrak{k}^\mathbb{C}\cap\mathfrak{m}^+$ which implies $N_s\cong N_{r-s}$. Also in this case $Adc_{\Pi}\mathfrak{k}^\mathbb{C}\cap\mathfrak{m}^+=0$ which implies $N_r\cong pt$.
\end{proposition}
\begin{proof}
	For each compact positive root $\phi$, if $\phi+\alpha_i$ is a positive noncompact root for some $\alpha_i\in \Pi$, then by Restricted Root theorem either (i) there exists a unique pair $\alpha_i,\alpha_j\in \Pi$ with $j\neq i$ such that $\alpha_i+\phi$ and $\alpha_j-\phi$ are noncompact positive roots or (ii) there exists one root $\alpha_i\in \Pi$ such that $\phi+\alpha_i$ is a noncompact positive root and $\phi$ is strongly orthogonal to other roots in $\Pi$. In case (i), $\phi+\alpha_i-\alpha_j$ is another compact  root and the adjoint action of $\exp(\frac{\pi}{2}(x_{\alpha_i}+x_{\alpha_j}))$ transforms the coroot $h_{\phi}$ into the coroot of $\phi+\alpha_i-\alpha_j$, also the adjoint action of $\exp(\frac{\pi}{2}x_{\alpha_t})(t\neq i,j)$ keeps $h_{\phi}$ invariant, by  Lemma \ref{inftra}. In case (ii),  the adjoint action of $\exp(\frac{\pi}{2}x_{\alpha_i})$ transforms the coroot $h_{\phi}$ into the coroot of $\phi+\alpha_i$ and the adjoint action of $\exp(\frac{\pi}{2}x_{\alpha_t})(t\neq i)$ keeps $h_{\phi}$ invariant, also by Lemma \ref{inftra}. Hence if the root vector $e_{\beta}$ is in $Adc_{\Gamma}\mathfrak{k}^\mathbb{C}\cap\mathfrak{m}^+$, there exists only one $\alpha_s\in \Gamma$ such that $\alpha_s-\beta$ is a root.
	
	Moreover when $X$ is of tube type, case (ii) does not exist. For any positive noncompact root $\beta\notin \Pi$, there exists only two roots $\alpha_{i}, \alpha_j\in \Pi$ such that $\alpha_i-\beta, \alpha_j-\beta$ are roots. Also we know $e_\beta\in Adc_{\Gamma}\mathfrak{k}^\mathbb{C}\cap\mathfrak{m}^+$ if and only if $\alpha_i\in \Gamma, \alpha_j\in \Pi-\Gamma$ or $\alpha_j\in \Gamma, \alpha_i\in \Pi-\Gamma$. Hence $Adc_{\Gamma}\mathfrak{k}^\mathbb{C}\cap\mathfrak{m}^+=Adc_{\Pi-\Gamma}\mathfrak{k}^\mathbb{C}\cap\mathfrak{m}^+$.
\end{proof}

\begin{proposition}
	We conclude the following table giving $N_k$ explicitly.
	\begin{table}[h]
		\centering
		\caption{The loci of infinity points}
		\begin{tabular}{|p{8em}|p{11em}|p{2em}|p{7em}|}
			\hline
			$X$ & $ N_k(1\leq k\leq r)\cong$ & $r$ & $X$ tube type ?\\  \hline
			$G(p,q) (p\geq q\geq 2)$ & $G(k,p-k)\times G(k,q-k)$  & $q$ & when $p=q$\\ \hline
			$G^{II}(n,n)(n\geq 4)$ & $G(2k,n-2k)$ & $\lfloor \frac{n}{2}\rfloor$ & when $n$ is even\\
			\hline
			$G^{III}(n,n)(n\geq 2)$ & $G(k,n-k)$ & $n$ & yes\\ \hline
			$Q^n(n\geq 3)$& $N_1\cong Q^{n-2}, N_2\cong pt$
			& 2 & yes\\ \hline
			$\mathbb{OP}^2$ &$N_1\cong G^{II}(5,5), N_2\cong Q^8$ & 2 & no\\
			\hline
			$E_7/P_7$ &$N_1\cong N_2\cong \mathbb{OP}^2, N_3\cong pt$ & 3 & yes\\
			\hline
		\end{tabular}
	\end{table}
\end{proposition}

\begin{proof}
	Type I,II,III cases can be obtained by matrices expression using Proposition \ref{inftypart}. For Type I case ($X=G(p,q)$),  we know $\mathfrak{m}^+\cong \{\begin{bmatrix}
	0 & A \\
	0 & 0
	\end{bmatrix}: A \in M(p,q,\mathbb{C})\}$. Then for $|\Gamma|=k$, $Adc_{\Gamma}\mathfrak{k}^\mathbb{C}\cap\mathfrak{m}^+$ is spanned the subspace of with $A$ in the form of $\begin{bmatrix}
	0 & \text{a $k\times (q-k)$ matrix} \\
	\text{a $(p-k)\times k$ matrix} & 0
	\end{bmatrix}$. For Type II, III cases, the argument is similar under the restriction $A=-A^T$, $A=A^T$ respectively.
	Also, $Q^n$ is of tube type and rank two.
	Therefore for classical cases $N_k$ can be easily obtained. For $E_7/P_7$ case, $N_1$ is isomorphic to the VMRT and by symmetry obtained in Proposition \ref{inftypart} we know $N_2\cong N_1$ is also isomorphic to the VMRT. For $\mathbb{OP}^2$, it suffices to compute $N_2$. 
	We adapt the simple roots convention as in Figure \ref{E6figure}.
	All the positive noncompact roots can be listed according to the coefficients of $\alpha_1,...,\alpha_6$ in order, which are
	\begin{flalign*}
	&(100000), (110000), (111000), (111001),
	(111100), (111101), (111110), (112101) \\
	&(111111), (122101), (112111), (122111),
	(112211), (122211), (123211), (123212)
	\end{flalign*}
	There is a $Q^8$ which is of sub-diagram type with positive noncompact roots 
	\begin{flalign*}
	&(100000), (110000), (111000), (111001),
	(111100), (111101), (112101), (122101)
	\end{flalign*}
	Roots whose corresponding root vectors in $Adc_{\Gamma}\mathfrak{k}^\mathbb{C}\cap\mathfrak{m}^+$ are given by 
	\begin{flalign*}
	& (111110), (111111), (112111), (122111),
	(112211), (122211), (123211), (123212)
	\end{flalign*}
	We know $N_2$ is also isomorphic to $Q^8$, by the adjoint action of $e_{\alpha_2+\alpha_3+\alpha_4+\alpha_5}$ on the $Q^8$ associated to the sub-diagram.
\end{proof}

\subsection{Stratifications of the centers}\label{strati}

By Lemma \ref{secantrank} we know $\mathcal{C}_o^{r-1}(X)\subsetneq \mathbb{P}T_o(X) \cong\mathbb{P}^{n-1}$ and $\mathcal{C}_o^{r}(X)=\mathbb{P}T_o(X)\cong \mathbb{P}^{n-1}$. Recall that $D$ is the compactifying divisor which is the complement of the affine cell $W$ in $X$. We inductively define $\mathcal{M}_1:=N_r$, \[\mathcal{M}_{k}:=\overline{\bigcup\{\text{mrcs passing} \ \mathcal{M}_{k-1}\}}\] where mrc means minimal rational curves for short. For simplicity we call $\mathcal{M}_k$ the $k$-th locus of chains of minimal rational curves starting from $\mathcal{M}_1$. We first observe that
\begin{lemma}
	 $\mathcal{M}_r=D$.
\end{lemma}
\begin{proof}
	This can be quickly deduced from the polysphere theorem. Since $K$-action on a polysphere exhausts $X$ and $K$-action keeps $o$ invariant, we know for a point $y\in D$, there exists an $r$-sphere $(\mathbb{P}^1)^r$ passing $o$ such that  $y\in D\cap (\mathbb{P}^1)^r=\bigcup_{1\leq i\leq r, \{\infty\} \ \text{is in the} \ i-\text{th factor}} \mathbb{P}^1\times \cdots \{\infty\} \cdots \times \mathbb{P}^1$. In the polysphere we write $o=(0,...,0), y=(y_1,...,y_n)$ where there is at least one $y_j=\infty$. $y$ can be connected to $N_r$ by minimal rational curves inside $D$, thus $y\in \mathcal{M}_r$. Hence $D\subset \mathcal{M}_r$. The other direction is trivial by definition, so $\mathcal{M}_r=D$.
\end{proof}

Moreover we have the following proposition on the stratifications $\mathcal{C}_o^1(X)\subset \cdots \subset \mathcal{C}_o^{r-1}(X)\subset \mathbb{P}T_o(X)\cong \mathbb{P}^{n-1}$ and $\mathcal{M}_1\subset \cdots \subset \mathcal{M}_{r-1} \subset D$ respectively.
\begin{proposition}\label{fiboverinfty}
	\begin{enumerate}
		\item
	$\mathcal{C}_o^k(X)\backslash \mathcal{C}_o^{k-1}(X)$ is a fiber bundle over $N_k$ with $d_{r-k+1}$-dimensional fibers, where $d_{r-k+1}+1$ is equal to dimension of the balanced subspace with rank $k$. The fiber bundle can be written as \[\theta_k:\mathcal{C}_o^k(X)\backslash \mathcal{C}_o^{k-1}(X)=K^{\mathbb{C}}/J_k\rightarrow N_k=K^{\mathbb{C}}/Q_k\]
	where $Q_k, J_k$ are the isotropy subgroups of the action of $K^{\mathbb{C}}$ on $N_k$ and $\mathcal{C}_o^k(X)\backslash \mathcal{C}_o^{k-1}(X)$ respectively. Moreover we have a $K^{\mathbb{C}}$-equivariant embedding $\mathcal{C}_o^k(X)\backslash \mathcal{C}_o^{k-1}(X)\hookrightarrow\mathcal{P}_k$ where $\mathcal{P}_k$ is the projectivization of the $Q_k$-homogeneous bundle over $N_k$, $K^{\mathbb{C}}\times_{Q_k}\mathbb{C}^{d_{r-k+1}+1}$. Let $W^o_k$ be a balanced subspace with rank $k$ passing through $o$ and $\mathcal{C}_o^{k-1}(W^o_k)$ is the $(k-2)$-th secant variety for the VMRT of $W^o_k$, then each fiber of $\theta_k$
	is isomprphic to  $\mathbb{P}T_o(W^o_k)\backslash\mathcal{C}^{k-1}_o(W^o_k)$, which is the
	complement of some degree $k$ hypersurface in a projective space $\mathbb{P}^{d_{r-k+1}}$.
	\item $N_{r-k+1}\subset \mathcal{M}_k$ and 
	$\mathcal{M}_k\backslash \mathcal{M}_{k-1}\cong K^{\mathbb{C}}\times_{Q_{r-k+1}}\mathbb{C}^{c_{r-k+1}}$ is a $Q_{r-k+1}$-homogeneous bundle over $N_{r-k+1}$ with rank $c_{r-k+1}$.
	It can be written as 
	\[\kappa_k :\mathcal{M}_k\backslash \mathcal{M}_{k-1}\rightarrow N_{r-k+1}\]
	where each fiber is Harish-Chandra embedding of $\mathbb{C}^{c_{r-k+1}}$ into a characteristric symmetric subspace of rank $k-1$ and the zero point is in $N_{r-k+1}$.
	
	\end{enumerate}
\end{proposition}	
\begin{proof}
	\begin{enumerate}
		\item 
	
	This is actually a corollary of Proposition \ref{tubefib}, we use it in advance. If we fix an infinity point $z\in N_k= K^{\mathbb{C}}/Q_k$, $\theta_k^{-1}(z)$ is the collection of projectivization of rank $k$ tangent vectors at the origin toward to $z$, i.e. for any $[v]\in\theta_k^{-1}(z)$, the limit of the corresponding affine line $\{\exp(tv).o, t\in \mathbb{C}\}$ is $z$. From Proposition \ref{tubefib} this coincidences with the set of projectivization of rank $k$ tangent vectors at the origin of the balanced subspace with rank $k$ passing $z$. From Proposition \ref{tubehyper}, the latter is the complement of some degree $k$ hypersurface in the corresponding  projective space. Also from Proposition \ref{inftypart} we know tangent space of the balanced subspace 
	can be identified with a $Q_k$-representation and this gives $\mathcal{P}_k$ and hence the embedding $\mathcal{C}_o^k(X)\backslash\mathcal{C}_o^{k-1}(X)\hookrightarrow\mathcal{P}_k$.
	
	\item  To write down a maximal polysphere we choose a maximal set of strongly orthogonal roots $\Pi\subset \Delta^+_{M}$ starting from the highest root $\alpha_1$. Let $\alpha_j$ be the highest element strongly orthogonal to $\{\alpha_1,\cdots, \alpha_{j-1}\}$. Then $\Pi=\{\alpha_1,\cdots,\alpha_r\}$ and the polysphere can be written as $\mathcal{P}=\{\prod_{1\leq i\leq r}\exp(t_ie_{\alpha_i}).o\}\subset X$ where $ t_i \in \mathbb{C}\cup\{\infty\}$ (when $t_i=\infty$, it means the corresponding factor is given by $\exp(\frac{\pi}{2}x_{\alpha_i})$). We know $\mathcal{P}\cap \mathcal{M}_k\backslash \mathcal{M}_{k-1}=\bigcup_{ \#\{i,t_i=\infty\}=r-k+1}(t_1,\cdots, t_n)$. Let $y_{r-k+1}$ be the point in $\mathcal{P}\cap N_{r-k+1}$ given by $t_1=\cdots=t_{r-k+1}=\infty, t_i=0 (i\geq r-k+2)$. Alternatively we may write
	\[y_{r-k+1}=c_{\Pi_{r-k+1}}.o\]
	where $\Pi_{r-k+1}=\{\alpha_1,\cdots, \alpha_{r-k+1}\}$.
	We then consider the $K^{\mathbb{C}}$-action on the reference point $y_{r-k+1}$ with isotropy group $Q_{r-k+1}$ and compute the action of $Q_{r-k+1}$ on the affine space in $\mathcal{P}\cap \mathcal{M}_{k}\backslash \mathcal{M}_{k-1}$ containing $y_{r-k+1}$ given by  \[\mathcal{Y}_{r-k+1}=\{c_{\Pi_{r-k+1}}\prod_{r-k+2\leq i\leq r}\exp(t_ie_{\alpha_i}).o, t_i\in \mathbb{C}\}\cong \mathbb{C}^{k-1}\]
	 Let $\Delta_C$ be the set of compact roots, i.e. the corresponding root spaces are contained in $\mathfrak{k}^{\mathbb{C}}$. From Lemma \ref{inftra} and the discussion in Proposition \ref{inftypart}
     we have
     \[Lie(Q_{r-k+1})=\mathfrak{h}^{\mathbb{C}}\oplus \mathbb{C}\{e_{\phi}, \phi\in \Delta_C , b_\phi=0 \ \text{or} \ 2\}\oplus \mathbb{C}\{e_{\phi}, \phi\in \Delta^-_C , b_\phi=1\}\]
  where $b_{\phi}$ is the number of roots in $\Pi_{r-k+1}$ not strongly orthogonal to $\phi$. In terms of restricted root, the roots with root vectors in $Lie(Q_{r-k+1})$ are corresponding to restricted roots of the form
  	\begin{itemize}
  	\item $\phi\in \Delta_C, b_{\phi}=0:  \phi|_{\mathfrak{h}^-}=\frac{1}{2}(\alpha_\ell-\alpha_j) ( r-k+2\leq \ell, j \leq r)$ or $\phi|_{\mathfrak{h}^-}=\pm\frac{1}{2}\alpha_\ell (r-k+2\leq \ell \leq r)$ (if exist) or $\phi|_{\mathfrak{h}^-}=0$ (if exist).
  	
  	\item $\phi\in \Delta_C, b_{\phi}=2: \phi|_{\mathfrak{h}^-}=\frac{1}{2}(\alpha_\ell-\alpha_j) ( 1\leq \ell, j \leq r-k+1)$.
  	
  	\item $\phi\in \Delta^-_C, b_{\phi}=1: \phi|_{\mathfrak{h}^-}=\frac{1}{2}(\alpha_\ell-\alpha_j) ( 1\leq \ell<r-k+2\leq j \leq r)$ or $\phi|_{\mathfrak{h}^-}=\frac{1}{2}\alpha_\ell (1\leq \ell \leq r-k+1)$ (if exist).
  \end{itemize}
 Moreover denoting the reductive part of $Lie(Q_{r-k+1})$ by $\mathfrak{R}(Lie(Q_{r-k+1}))$ we have 
 \[\mathfrak{R}(Lie(Q_{r-k+1}))=\mathfrak{h}^{\mathbb{C}}\oplus \mathbb{C}\{e_{\phi}, \phi\in \Delta_C , b_\phi=0 \ \text{or} \ 2\}\supset \mathfrak{k}_1\oplus \mathfrak{k}_2 \]
 where $\mathfrak{k}_1, \mathfrak{k}_2$ are  reductive parts for the parabolic subalgebras associated to the balanced subspace with rank $r-k+1$, and the Hermitian characteristic symmetric subspace of rank $k-1$.
 
  From Lemma \ref{inftra} we know $Adc_{\Pi_{r-k+1}}$ stabilizes $\mathfrak{h}^{\mathbb{C}}$ and then the action from $\exp(\mathfrak{h}^{\mathbb{C}})$ keeps $\mathcal{Y}_{r-k+1}$ invariant. A root $\phi$ in $\Delta_C$ with $b_{\phi}=0$ satisfies $Adc_{\Pi_{r-k+1}}e_{\phi}=e_{\phi}$. For a root $\phi$ in $\Delta_C$ with $b_{\phi}=2$, up to nonzero constant multiplication we know $Adc_{\Pi_{r-k+1}}e_{\phi}$ is equal to  $e_{\phi+\alpha_j-\alpha_\ell}$ and hence the action $\exp(e_{\phi})$ given by such $e_{\phi}$ keeps $\mathcal{Y}_{r-k+1}$ invariant. For a root $\phi\in \Delta^-_C$ with $\phi|_{\mathfrak{h}^-}=\frac{1}{2}(\alpha_\ell-\alpha_j) ( 1\leq \ell<r-k+2\leq j \leq r)$ or $\phi|_{\mathfrak{h}^-}=\frac{1}{2}\alpha_\ell (1\leq \ell \leq r-k+1)$, we know 
  \[\begin{aligned}
 &\exp(e_{\phi})c_{\Pi_{r-k+1}}\prod_{r-k+2\leq i\leq r}\exp(t_ie_{\alpha_i}).o =c_{\Pi_{r-k+1}}\exp(e_{\phi-\alpha_\ell})\prod_{r-k+2\leq i\leq r}\exp(t_ie_{\alpha_i}).o \\
  &=c_{\Pi_{r-k+1}}\prod_{r-k+2\leq i\leq r}\exp(e^{ade_{\phi-\alpha_\ell}}t_ie_{\alpha_i}).o\\&=c_{\Pi_{r-k+1}}\left(\prod_{r-k+2\leq i\leq r, i\neq j}\exp(t_ie_{\alpha_i})\right) \exp(t_je_{\alpha_j}+t'e_{\phi+\alpha_j-\alpha_{\ell}}+t''e_{-\alpha_\ell}).o \ \\&(t'=t''=0 \ \text{if} \ \phi|_{\mathfrak{h}^-}=\frac{1}{2}\alpha_\ell )\\
    &=c_{\Pi_{r-k+1}}\prod_{r-k+2\leq i\leq r}\exp(t_ie_{\alpha_i}).o  \ (\text{as} \ [e_{\alpha_j}, t'e_{\phi-\alpha_{\ell}+\alpha_j}+t''e_{-\alpha_\ell}]=0)
  \end{aligned}
 \]
   Thus the action $\exp(e_{\phi})$ given by such $e_{\phi}$ keeps $\mathcal{Y}_{r-k+1}$ invariant. In conclusion we can obtain that nontrivial action just comes from root vectors with roots $\phi\in \Delta_C$ with $b_{\phi}=0$ and hence
\[    
  Q_{r-k+1}.\mathcal{Y}_{r-k+1}\subset \{c_{\Pi_{r-k+1}}\exp(v).o, v\in \mathfrak{m}_{r-k+1}^+\}   \ (\dagger)
  \]
  where $\mathfrak{m}_{r-k+1}^+=\mathbb{C}\{e_{\beta}, \beta|_{\mathfrak{h}^-}=\frac{1}{2}(\alpha_{\ell}+\alpha_j)\ \text{or} \ \frac{1}{2}\alpha_{\ell}  (r-k+2\leq \ell \leq j\leq r) \}$ is precisely tangent space of the Hermitian characteristic symmetric subspace. Also we know the s $\mathfrak{R}(Lie(Q_{r-k+1}))\supset \mathfrak{k}_1\oplus \mathfrak{k}_2 $, therefore the other inclusion in $(\dag)$ holds. Thus 
  \[ Q_{r-k+1}.\mathcal{Y}_{r-k+1}= \{c_{\Pi_{r-k+1}}\exp(v).o, v\in \mathfrak{m}_{r-k+1}^+\}\]
  By $K^\mathbb{C}$-action on the polysphere passing $o$, we know $\mathcal{M}_k\backslash \mathcal{M}_{k-1}$ can be exhausted and 
  $\mathcal{M}_k\backslash \mathcal{M}_{k-1}=K^{\mathbb{C}}\times_{Q_{r-k+1}}\mathbb{C}^{c_{r-k+1}}$ is a $Q_{r-k+1}$ homogeneous bundle over $N_{r-k+1}$ with rank $c_{r-k+1}$. Also from the discussion we can see each fiber is Harish-Chandra embedding of $\mathbb{C}^{c_{r-k+1}}$ into a Hermitian characteristric symmetric subspace of rank $(k-1)$ and the zero point is in $N_{r-k+1}$.
\end{enumerate}
\end{proof}

On the other hand, we construct the following subvarieties which serves as the loci of chains of minimal rational curves.
Let $\mathcal{V}_0=\{o\}$. Let 
\[\mathcal{V}_{k}=\overline{\bigcup\{\text{mrcs passing} \ \mathcal{V}_{k-1}\}}.\]  
We know $\mathcal{V}_{k-1}\subset \mathcal{V}_k$. General properties on geometry of loci of chains of mininal rational curves were discussed in \cite{MR2155089}. Note that in our construction for $\mathcal{M}_k$, we start from the subvariety $N_r$ while in the construction for $\mathcal{V}_k$, we start from the base point $o$. The latter is precisely the same as the construction in \cite{MR2155089}.

 As in Proposition \ref{fiboverinfty}(2) we have

\begin{proposition}\label{tubefib}
	$N_{k}\subset \mathcal{V}_k$ and 
	$\mathcal{V}_k\backslash \mathcal{V}_{k-1}\cong K^{\mathbb{C}}\times_{Q_{k}}\mathbb{C}^{d_{r-k+1}+1}$ is a $Q_{k}$-homogeneous bundle over $N_{k}$ with rank $d_{r-k+1}+1$.
	It can be written as 
	\[\nu_k :\mathcal{V}_k\backslash \mathcal{V}_{k-1}\rightarrow N_{k}\]
	where each fiber is Harish-Chandra embedding of $\mathbb{C}^{d_{r-k+1}+1}$ into a balanced subspace with rank $k$ and the zero point is in $N_{k}$. Therefore $\mathcal{V}_k$ is the variety swept by all balanced subspaces with rank $k$ passing through $o$.
\end{proposition}
\begin{proof}
	The proof is parallel to the proof of Proposition \ref{fiboverinfty}(2).
	We still write the maximal polysphere as $\mathcal{P}=\{\prod_{1\leq i\leq r}\exp(t_ie_{\alpha_i}).o\}\subset X$ where $ t_i \in \mathbb{C}\cup\{\infty\}$ (when $t_i=\infty$, it means the corresponding factor is given by $\exp(\frac{\pi}{2}x_{\alpha_i})$). We know $\mathcal{P}\cap \mathcal{V}_k\backslash \mathcal{V}_{k-1}=\bigcup_{ \#\{i,t_i=0\}=r-k}(t_1,\cdots, t_n)$.
	Let $z_k$ be the point in $\mathcal{P}\cap N_{k}$ given by $t_1=\cdots=t_{k}=\infty, t_i=0 (i\geq k+1)$. Alternatively we may write
	\[z_k=c_{\Pi_{k}}.o\]
	where $\Pi_k=\{\alpha_1, \cdots, \alpha_k\}$.
	We then consider the $K^{\mathbb{C}}$-action on the reference point $z_{k}$ with isotropy group $Q_{k}$ and compute the action of $Q_{k}$ on the affine space in $\mathcal{P}\cap \mathcal{V}_{k}\backslash \mathcal{V}_{k-1}$ containing $z_k$ given by  \[\mathcal{Z}_{k}=\{c_{\Pi_{k}}\prod_{1\leq i\leq k}\exp(t_ie_{\alpha_i}).o, t_i\in \mathbb{C}\}\cong \mathbb{C}^k\]
	Recall that $\Delta_C$ is the set of compact roots. Also 
	we have
	\[Lie(Q_{k})=\mathfrak{h}^{\mathbb{C}}\oplus \mathbb{C}\{e_{\phi}, \phi\in \Delta_C , b_\phi=0 \ \text{or} \ 2\}\oplus \mathbb{C}\{e_{\phi}, \phi\in \Delta^-_C , b_\phi=1\}\]
	where $b_{\phi}$ is the number of roots in $\Pi_{k}$ not strongly orthogonal to $\phi$. 
	Using similar argument as in the proof of Proposition \ref{fiboverinfty}(2) we know 
	nontrivial action just comes from root vectors with roots $\phi\in \Delta_C$ with $b_{\phi}=2$ and also
	\[    
	Q_{k}.\mathcal{Z}_{k}= \{c_{\Pi_{k}}\exp(v).o, v\in \mathfrak{u}_{k}^+\}   
	\]
	where $\mathfrak{u}_{k}^+=\mathbb{C}\{e_{\beta}, \beta|_{\mathfrak{h}^-}=\frac{1}{2}(\alpha_{\ell}+\alpha_j), (1\leq \ell \leq j\leq k) \}$ is precisely tangent space of the balanced subspace at $o$.
	
	By $K^\mathbb{C}$-action on the polysphere passing $o$, we know $\mathcal{V}_k\backslash \mathcal{V}_{k-1}$ can be exhausted and 
	$\mathcal{V}_k\backslash \mathcal{V}_{k-1}=K^{\mathbb{C}}\times_{Q_{k}}\mathbb{C}^{d_{r-k+1}+1}$ is a $Q_k$ homogeneous bundle over $N_{k}$ with rank $d_{r-k+1}+1$. Also
	 we can see each fiber is  Harish-Chandra embedding of $\mathbb{C}^{d_{r-k+1}+1}$ into a balanced subspace with rank $k$ and the zero point is in $N_{k}$.
\end{proof}
\begin{corollary}\label{transversal}
	For any point $x\in \mathcal{M}_k\backslash \mathcal{M}_{k-1}$,
there is a balanced subspace with rank $r-k+1$ passing $x$ which is transversal to $\mathcal{M}_{k}$ at $x$.
\end{corollary}
\begin{proof}
	We first consider a point $x\in N_{r-k+1} \subset \mathcal{M}_k$. This is straightforward from the proof of Proposition \ref{inftypart}, Proposition \ref{fiboverinfty}(2) and  Proposition \ref{tubefib}. More precisely, without loss of generality we may assume that $x=c_{\Pi_{r-k+1}}.o$ where $\Pi_{r-k+1}=\{\alpha_1, \cdots, \alpha_{r-k+1}\}$. From Proposition \ref{fiboverinfty}(2) and Proposition \ref{tubefib} we know there is a Hermitian characteristic symmetric subspace of rank $k-1$ and a balanced subspace of rank $r-k+1$ passing $x$, which are denoted by $Z_{k-1}^x, W^x_{r-k+1}$ respectively. If suffices to show that $T_x(Z_{k-1}^x)+T_x(W_{r-k+1}^x)+T_x(N_{r-k+1})=T_x(X)$. While we have 
	\[\begin{cases}
		T_o(c_{\Pi_{r-k+1}}Z_{k-1}^x)=\mathfrak{m}_{r-k+1}^+=\mathbb{C}\{e_{\beta}, \beta|_{\mathfrak{h}^-}=\frac{1}{2}(\alpha_{\ell}+\alpha_j)\ \text{or} \ \frac{1}{2}\alpha_{\ell}  (r-k+2\leq \ell \leq j\leq r) \} \\
		T_o(c_{\Pi_{r-k+1}}W_{r-k+1}^x)=\mathfrak{u}_{r-k+1}^+=\mathbb{C}\{e_{\beta}, \beta|_{\mathfrak{h}^-}=\frac{1}{2}(\alpha_{\ell}+\alpha_j), (1\leq \ell \leq j\leq r-k+1) \} \\
		T_o(c_{\Pi_{r-k+1}}N_{r-k+1})=Adc_{\Pi_{r-k+1}}\mathfrak{k}^{\mathbb{C}}\cap \mathfrak{m}^+=\mathbb{C}\{e_{\beta}, \beta|_{\mathfrak{h}^-}=\frac{1}{2}(\alpha_{\ell}+\alpha_j)\ \text{or} \ \frac{1}{2}\alpha_{\ell}  (1\leq \ell \leq r-k+1< j\leq r) \} 
	\end{cases}
	\]
	where the last equality is from Proposition \ref{inftypart}.
	Now applying the Restricted Root Theorem we have $T_o(X)=T_o(c_{\Pi_{r-k+1}}Z_{k-1}^x)+T_o(c_{\Pi_{r-k+1}}W_{r-k+1}^x)+T_o(c_{\Pi_{r-k+1}}N_{r-k+1})$ identifying $T_o(X)$ with $\mathfrak{m}^+$. Then we go back to $x$ the tangent space equality $T_x(Z_{k-1}^x)+T_x(W_{r-k+1}^x)+T_x(N_{r-k+1})=T_x(X)$ still holds. 
	Now for any $x\in \mathcal{M}_k\backslash \mathcal{M}_{k-1}$, we know for $\kappa_k(x)\in N_{r-k+1}$ there is such a balanced subspace $W^{\kappa_k(x)}_{r-k+1}$. Also we know $x$ can be translated from $\kappa_k(x)$ by a group action $g\in Aut(Z^{\kappa(x)}_{k-1})$, then the balanced subspace $g.W^{\kappa_k(x)}_{r-k+1}$ will be the desired one.
\end{proof}

\begin{remark}
	From Proposition \ref{tubefib} we can see that $N_k$ parametrizes all balanced subspaces with rank $k$ passing $o$ and this gives an explanation for the Tits fibartion associated to the pair of subdiagram type in this special case. For example, for the balanced subspace $Q^8\subset \mathbb{OP}^2$, adapting the convention in Figure \ref{E6figure} we know in Tits fibration the parameter space of such $Q^8$ passing $o$ is given by deleting the marked node $\alpha_1$ and marking the node adjacent to the subdiagram, which is  $\alpha_5$. This is precisely $N_2\cong Q^8$.
\end{remark}


\begin{remark}
	There may exist more perspectives for the stratification constructed on $X$. We give one more example here. We know the compactifying divisor $D\subset X$ is actually 
	a section $s$ of the ample generator of $Pic(X)$ which vanishes to order $r$ at some point. Let $V_j(s)$ be the set of points at which the vanishing order of $s$ is at least $j$, which is precisely the multiplicity subspaces defined in \cite{MR1115788}. We can see that $V_j(s)=\mathcal{M}_{r-j+1}$.
\end{remark}

\subsection{Singularities and smoothing the stratification}
For the stratification, we describe the singular loci for the subvarieties.
\begin{proposition}
	\begin{enumerate}
		\item The singular locus of $\mathcal{C}_o^k(X)$ is 
	$Sing(\mathcal{C}_o^{k}(X))=\mathcal{C}_o^{k-1}(X) (2\leq k\leq r-1)$;
	\item The singular locus of $\mathcal{M
	}_k$ is
	$Sing(\mathcal{M}_{k})=\mathcal{M}_{k-1} (2\leq k\leq r)$.	
\end{enumerate}
\end{proposition}
\begin{proof}
	\begin{enumerate}
		\item 
	 We know $\mathcal{C}_o^{k}(X)\backslash \mathcal{C}_o^{k-1}(X)$ is a Zariski open dense smooth $P$-orbit, so $Sing(\mathcal{C}_o^{k}(X))\subset \mathcal{C}_o^{k-1}(X)$. Also from Zak \cite[Lemma 1.9 a)]{MR1234494}, we have $\mathcal{C}_o^1(X)\subset Sing(\mathcal{C}_o^2(X))$ since $\mathcal{C}_o^1(X)$ is nondegenerate in $\mathbb{P}T_o(X)$. Hence $\mathcal{C}_o^1(X)=Sing(\mathcal{C}_o^2(X))$ and we are done with cases where $\text{rank}(X)\leq 3$.
	 One way to obtain the inclusion $Sing(\mathcal{C}_o^{k}(X))\subset \mathcal{C}_o^{k-1}(X)$ in general is to consider case by case and do the checking on radical ideals (for exampel in Grassmannian case given by the minors). 
	We give a uniform proof using projective geometry here. In fact we can show that $ \hat{T}_\alpha(\mathcal{C}_o^k(X))=\mathbb{P}^{n-1}$ for $\alpha\in \mathcal{C}_o^{k-1}(X)$ where $\hat{T}_\alpha(\mathcal{C}_o^k(X))$ denotes the embedded projective tangent space. This can be done by induction. We can see $Sing(\mathcal{C}_o^2(X))=\mathcal{C}_o^1(X)$ and  $ \hat{T}_\alpha(\mathcal{C}_o^2(X))=\mathbb{P}^{n-1}$ for $\alpha\in \mathcal{C}_o^1(X)$ from \cite[Lemma 1.9 a)]{MR1234494} as $\mathcal{C}_o^1(X)$ is nondegenerate. Now suppose that $ \hat{T}_\alpha(\mathcal{C}_o^{k_0}(X))=\mathbb{P}^{n-1}$ for $\alpha\in \mathcal{C}_o^{k_0-1}(X)$ and $Sing(\mathcal{C}_o^{k_0}(X))=\mathcal{C}_o^{k_0-1}(X)$. For $\beta\in \mathcal{C}_o^{k_0}(X)=S(\mathcal{C}_o^{k_0-1}(X), \mathcal{C}_o^1(X))$, choose $\alpha\in \mathcal{C}_o^{k_0-1}(X), \gamma\in \mathcal{C}_o^1(X)$ ($\gamma\neq \alpha$) such that $\beta\in <\alpha,\gamma>$. Then 
	$\hat{T}_\beta(\mathcal{C}_o^{k_0+1}(X))\supset span<\hat{T}_{\alpha}(\mathcal{C}_o^{k_0}(X)),\hat{T}_{\gamma}(\mathcal{C}_o^1(X))>$ and thus $\hat{T}_\beta(\mathcal{C}_o^{k_0+1}(X))=\mathbb{P}^{n-1}$.
	
	\item From Proposition \ref{fiboverinfty} we know $Sing(\mathcal{M}_k)\subset \mathcal{M}_{k-1}$. To obtain the other inclusion 
	we claim that $\dim T_x(\mathcal{M}_k)=\dim(X)=n$ if $x\in \mathcal{M}_{k-1}$. By simple induction it suffices to consider the case when $x\in \mathcal{M}_{k-1}\backslash \mathcal{M}_{k-2}$. By Proposition \ref{fiboverinfty} we may further assume that $x\in N_{r-k+2}$. From the construction of the stratification there exists a family of lines passing $x$ and $N_{r-k+1} $ contained in $\mathcal{M}_k$. This is a cone over a subvariety $\mathfrak{N}\subset N_{r-k+1}$. From Corollary \ref{transversal} we know there is a balanced subspace with rank $r-k+2$ passing $x$ which is transversal to $\mathcal{M}_{k-1}$ at $x$, and we use $W^x_{r-k+2}$ to denote it. Then we know  $\mathfrak{N}=N^{W_{r-k+2}}_{r-k+1}:=N_{r-k+1}\cap W^x_{r-k+2}\cong \mathcal{C}_x^1(W^x_{r-k+2})$ where $\mathcal{C}_x^1(W^x_{r-k+2})$ is the VMRT of 
	$W^x_{r-k+2}$. As the affinization of $\mathcal{C}_x^1(W_{r-k+2})$ spans $T_x(W_{r-k+2})$ and $W_{r-k+2}$ is transversal to $\mathcal{M}_{k-1}$ at $x$ we have $\dim T_x(\mathcal{M}_k)=\dim \mathcal{M}_{k-1}+\dim W_{r-k+2}=n$. Hence $Sing(\mathcal{M}_k)=\mathcal{M}_{k-1}$.

    \end{enumerate}
\end{proof}
Now the geometric picture is clear. $\mathcal{M}_k$ is the union of a family of Hermtian characteristic symmetric subspaces in $X$ parametrized by $N_{r-k+1}$, where those affine cells are disjoint in $\mathcal{M}_k\backslash \mathcal{M}_{k-1}$ and their compactifying divisors intersect inside $\mathcal{M}_{k-1}$. For example when $X=Q^n$, $D=\mathcal{M}_2$ is a cone over $N_1\cong Q^{n-2}$ (see Example \ref{examhyper}) with vertex in $\mathcal{M}_1\cong \{pt\}$. In this case a Hermitian characterisc symmetric subspace of rank one is just a minimal rational curve  in $Q^n$ and we can see that $\mathcal{M}_2$ is a family of $\mathbb{P}^1$'s where the affine cells ($\cong\mathbb{C}$) are disjoint in $\mathcal{M}_2\backslash \mathcal{M}_1$ and their compactifying divisors ($\cong \{pt\}$) intersects as the vertex.

Next we will prove that the successive blow-ups actually smooth the stratifications.
First we recall the Blow-up Closure Lemma.

\begin{lemma}[Blow-up Closure Lemma]\label{blowclosure}
	Suppose that $Y\hookrightarrow X$ is a closed subscheme corresponding to a finite type quasicoherent sheaf of ideals and $Z\hookrightarrow X$ is a closed subscheme. If we blow up $X$ along $Y$, then the strict transform of $Z$ is canonically isomorphic to $\text{Bl}_{Z\cap Y}Z$.
\end{lemma}

Also we recall the following lemma, see for example \cite[Chapter III, Corollary  1.4.5]{MR954831}.

\begin{lemma}\label{excdivnor}
	Suppose that $Y\hookrightarrow X$ is a closed subscheme corresponding to a finite type quasicoherent sheaf of ideals. If we blow up $X$ along $Y$, the exceptional divisor is the projectivized normal cone of $Y$ in $X$.
	\end{lemma}
 For simplicity we will use $NorCone(\mathcal{A},\mathcal{B})$ to denote the projectivized normal cone of $\mathcal{A}\subset \mathcal{B}$ and $NorCone(\mathcal{A},\mathcal{B})_x$ to denote the fiber at $x\in \mathcal{A}$  where $\mathcal{A}, \mathcal{B}$ are projective varieties. One can find the formal definition of normal cone and projectivized normal cone in \cite[pp.571-572]{MR954831} in terms of commutative algebra. When $\mathcal{A}$ is a point $x\in \mathcal{B}$, the projectivized normal cone is called the projectivized tangent cone of $\mathcal{B}$ at $x$, which is denoted by $TanCone(\mathcal{B},x)$. We given a geometric description of $NorCone(\mathcal{A},\mathcal{B})_x$ here when $x$ is in the smooth locus of $\mathcal{A}$: the projectivized normal cone at $x$ is the projectivization of the union of tangents to holomorphic arcs $\gamma: \Delta\rightarrow \mathcal{B}$ centered about $x$ modulo  $T_x(\mathcal{A})$. If we embed $\mathcal{B}$ into $\mathbb{P}^n$, then $NorCone(\mathcal{A},\mathcal{B})_x$ is a subvariety in $\mathbb{P}(N_{\mathcal{A}|\mathbb{P}^n,x})$.
 
 We can prove the following proposition.
 \begin{proposition}\label{singularityafterblowup}
 	Let $X^n$ be a singular projective variety with $Y^m\subset Z:=Sing(X^n)$ where $Y^m$ is a smooth variety. Let $NorCone(Y,X)\rightarrow Y$ be the  projectivized normal cone of $Y\subset X$. Let $\pi: Bl_YX \rightarrow X$ be the blow-up of $X$ along $Y$ and $\widetilde{Z}$ is the strict transform of $Z$. Then $Sing(Bl_YX)\subset  \widetilde{Z}\cup \bigcup_{y\in Y}Sing(NorCone(Y,X)_y)$ if we identify the exceptional divisor with $NorCone(Y,X)$. In particular, if for each $y\in Y$, $Sing(NorCone(Y,X)_y)=NorCone(Y,Z)_y$, then $Sing(Bl_YX)= \widetilde{Z}$.
 \end{proposition}
 
\begin{proof}
		It suffices to consider it locally. As $Y^m$ is smooth, by choosing suitable coordinate we may assume that $X^n$ is a defined by the polynomials $(f_1,\cdots ,f_e)$ in $\mathbb{C}^N=\{(x_1,\cdots, x_m, y_1, \cdots,y_d)\}$ and $Y^m=\{(x_1,\cdots, x_m, 0, \cdots,0)\}$. We can write down $\pi:Bl_{Y}\mathbb{C}^N\rightarrow \mathbb{C}^N$ and then the strict transform of $\widetilde{X}\rightarrow X$ in coordinates. First of all, we know
	\[Bl_{Y}\mathbb{C}^N=\{(x_1,\cdots, x_m, y_1,\cdots, y_d)\times [w_1,\cdots, w_d]: y_iw_j=y_jw_i\} \subset \mathbb{C}^N\times \mathbb{P}^{d-1}     \]
	For $f_s(1\leq s\leq e)$, we can write it as 
	\[f_s=f_{s,k_s}+f_{s,k_s+1}+\cdots\]where $f_{s,k_s}$ is a polynomial in $(x_1,\cdots, x_m, y_1, \cdots,y_d)$  whose order in $(y_1, \cdots, y_d)$ is $k_s$. Then in $\{w_1=1\}$ coordinate, we have $y_j=y_1w_j$ and $\pi^{-1}(X)$ is given by the following polynomials $(1\leq s\leq e)$
	\[f'_s=y^{k_s}_1f_{s,k_s}(x_1,\cdots, x_m,1,w_2,\cdots, w_d)+y_1^{k_s+1}f_{s,k_s+1}(x_1,\cdots, x_m,1,w_2,\cdots, w_d)+\cdots\] 
	and then $\widetilde{X}$ is given by the following polynomials $(1\leq s\leq e)$
	\[\widetilde{f_s}=f_{s,k_s}(x_1,\cdots, x_m, 1, w_2,\cdots, w_d)+y_1f_{s,k_s+1}(x_1,\cdots, x_m, 1, w_2,\cdots, w_d)+\cdots\] 
	and the exceptional divisor of blowing up $X$ along $Y$ is \[E=\{f_{s,k_s}(x_1,\cdots, x_m, 1, w_2,\cdots, w_d)=0 (1\leq s\leq e)\}.\] 
	We know the exceptional divisor $E$ is identified with $NorCone(Y,X)$. If we consider the fiber of $NorCone(Y,X)$ on $y=(x^0_1,\cdots, x^0_m)\in Y$, it is given by 
	\[NorCone(Y,X)_y=\{f_{s,k_s}(x^0_1,\cdots, x^0_m, 1, w_2,\cdots, w_d)=0 (1\leq s\leq e)\} \subset \mathbb{C}^{d-1}\] 
	On the smooth locus of $NorCone(Y,X)_y$, we know $\{d_{(w_2,\cdots, w_d)}f_{s,k_s}(x_1,\cdots, x_m,1,w_2,\cdots, w_d)\}_{1\leq s\leq e}$ is of rank $N-n$. On the other hand we know 
	\[d\tilde{f_s}|_{\{y_1=0\}}=df_{s,k_s}(x_1,\cdots, x_m, 1, w_2,\cdots, w_d)+f_{s,k_s+1}(x_1,\cdots, x_m, 1, w_2,\cdots, w_d)dy_1\] 
	Then we can see $\{d\widetilde{f_s}\}_{1\leq s\leq e}$ over $sm(NorCone(Y,X)_y)$ is of rank $N-n$ as well. Thus $Sing(Bl_YX)\subset  \widetilde{Z}\cup \bigcup_{y\in Y}Sing(NorCone(Y,X)_y)$.
\end{proof}
 
Next we consider the fiber of the projectivized normal cone $NorCone(\mathcal{M}_j, \mathcal{M}_k), NorCone(\mathcal{C}_o^j(X)), \mathcal{C}_o^k(X)$ over the smooth locus of $\mathcal{M}_j, \mathcal{C}_o^j(X)$ respectively. We have
\begin{proposition}\label{describnor}
 \begin{enumerate}
 	\item Let $y\in sm(\mathcal{M}_{j})=\mathcal{M}_j\backslash \mathcal{M}_{j-1}$ be a smooth point in $\mathcal{M}_j$. Then $NorCone(\mathcal{M}_j, X)_{y}$ can be identified with the projectivized tangent space of a balanced subspace $W^y_{r-j+1}$ of rank $r-j+1$ passing $y$ and $NorCone(\mathcal{M}_j, \mathcal{M}_k)_{y}\cong \mathcal{C}_y^{k-j+1}(W^y_{r-j+1})$. Moreover we have \[(NorCone(\mathcal{M}_j, \mathcal{M}_k)_{y}\subset  NorCone(\mathcal{M}_j, X)_{y})\cong (\mathcal{C}_y^{k-j+1}(W^y_{r-j+1})\subset \mathbb{P}T_y(W^y_{r-j+1}))\] where the isomorphism means that the embedding is the same.
 	\item Let  $v\in sm(\mathcal{C}_o^{j}(X))=\mathcal{C}
 	_o^j(X)\backslash \mathcal{C}_o^{j-1}(X)$ be a smooth point in $\mathcal{C}_o^j(X)$. Then $NorCone(\mathcal{C}_o^j(X), \mathbb{P}T_o(X))_v$ can be identified with the projectivized tangent space of a Hermitian characteristic symmetric subspace $Z^o_{r-j}$ passing $o$ of rank $r-j$ and $NorCone(\mathcal{C}_o^j(X), \mathcal{C}_o^k(X)_{v}\cong \mathcal{C}_o^{k-j}(Z^o_{r-j})$.
 	Moreover we have \[(NorCone(\mathcal{C}_o^j(X), \mathcal{C}_o^k(X)_{v}\subset  NorCone(\mathcal{C}_o^j(X), \mathbb{P}T_o(X))_v)\cong (\mathcal{C}_o^{k-j}(Z^o_{r-j})\subset \mathbb{P}T_o(Z^o_{r-j}))\] where the isomorphism means that the embedding is the same.
 \end{enumerate}
\end{proposition}

Before the proof we introduce the following construction.
Let $\mathcal{V}^{W^y_{r-j+1}}_s$ be the $s$-th locus of chains of minimal rational curves in $W^y_{r-j+1}$ starting from $y$, i.e. we inductively define $\mathcal{V}^{W^y_{r-j+1}}_{0}=\{y\}$ and
\[\mathcal{V}^{W^y_{r-j+1}}_{s}=\overline{\bigcup\{\text{mrcs in}\  W^y_{r-j+1} \ \text{passing} \ \mathcal{V}^{W^y_{r-j+1}}_{s-1}\}}.\] 
From \cite[Section 6.2]{MR2155089}, we have 
\begin{lemma}\label{tangconihss}
	$TanCone(\mathcal{V}^{W^y_{r-j+1}}_{s},y)=\mathcal{C}^s_y(W^y_{r-j+1}) \subset \mathbb{P}T_y(W^y_{r-j+1}) (1\leq s\leq r-j+1)$ where $\mathcal{C}^1_y(W^y_{r-j+1})$ is the VMRT of $W^y_{r-j+1}$ at $y$ and $\mathcal{C}^s_y(W^y_{r-j+1})$ is the $(s-1)$-th secant variety of $\mathcal{C}^1_y(W^y_{r-j+1})$.
\end{lemma} 

\begin{proof}[Proof of Proposition \ref{describnor}]
	\begin{enumerate}
		\item This follows immediately from Lemma \ref{tangconihss}. As for each $y\in \mathcal{M}_j\backslash \mathcal{M}_{j-1}$. By Corollary \ref{transversal} there is a balanced subspace $W^y_{r-j+1}$ of rank $r-j+1$ passing $y$ which is transversal to $\mathcal{M}_j\backslash \mathcal{M}_{j-1}$ at $y$.
		\item From Proposition-Definition \ref{HSS-normal} without loss of generality we may assume that $v=[e_{\alpha_1}+\cdots+e_{\alpha_j}]\in \mathbb{P}T_o(X)$ where $\alpha_1, \cdots, \alpha_j\in \Pi$ and the bracket means projectivization. Then from Proposition \ref{fiboverinfty}(1) and the Restricted Root Theorem we know the embedded projective tangent space $\hat{T}_v(\mathcal{C}_o^j(X)\backslash \mathcal{C}_o^{j-1}(X))=\mathbb{P}\widetilde{T_v}(\mathcal{C}_o^j(X)\backslash \mathcal{C}_o^{j-1}(X))$ where 
	\[	\widetilde{T_v}(\mathcal{C}^j_o(X)\backslash \mathcal{C}_o^{j-1}(X))=\mathbb{C}\{e_{\beta}, \beta|_{\mathfrak{h}^-}=\frac{1}{2}(\alpha_{\ell}+\alpha_s)\ \text{or} \ \frac{1}{2}\alpha_{\ell}  (1\leq \ell \leq j, 1\leq s\leq r) \}
	\]
	Then the quotient $T_o(X)/\widetilde{T_v}(\mathcal{C}^j_o(X)\backslash \mathcal{C}_o^{j-1}(X))$ can be  identified with \[\mathbb{C}\{e_{\beta}, \beta|_{\mathfrak{h}^-}=\frac{1}{2}(\alpha_{\ell}+\alpha_s)\ \text{or} \ \frac{1}{2}\alpha_{\ell}  (j< \ell\leq s\leq r)\}\] which is the tangent space of a Hermitian characteristic symmetric subspace $Z^o_{r-j}$ passing $o$. Hence $NorCone(\mathcal{C}_o^j(X), \mathbb{P}T_o(X))\cong \mathbb{P}T_o(Z^o_{r-j})$. Also we know 
	for $v'\in T_o(Z_{r-j}^o)$, $[e_{\alpha_1}+\cdots+e_{\alpha_k}+v']\in \mathcal{C}_o^k(X)$ if and only if $[v']\in \mathcal{C}_o^{k-j}(Z^o_{r-j})\subset \mathbb{P}T_o(Z^o_{r-j})$. Then the conclusion follows easily from the geometric description of projectivized normal cones.
	\end{enumerate}
\end{proof}
Then we obtain the smoothness of the centers.
\begin{proposition}\label{stofcha}
	\begin{enumerate}
		\item The successive strict transform 
	$\mu_{k-1}(\mathcal{M}_{k})$ is a desingularization of $\mathcal{M}_{k}$.
	
	\item The successive strict transform  $\tau_{k-1}(\mathcal{C}_o^k(X))$ is a desingularization of $\mathcal{C}_o^k(X)$.
	\end{enumerate}
\end{proposition}

\begin{proof}
	\begin{enumerate}
		\item 
		By Proposition \ref{describnor}(1) we know $Sing(NorCone(\mathcal{M}_1,\mathcal{M}_{k})_y)=NorCone(\mathcal{M}_1,\mathcal{M}_{k-1})_y$ for $y\in \mathcal{M}_1$.
		Since $\mathcal{M}_1$ is smooth, from Proposition \ref{singularityafterblowup} we obtain that $\mu_1(\mathcal{M}_2)$ is smooth and $Sing(\mu_1(\mathcal{M}_k))=\mu_1(\mathcal{M}_{k-1}))$. Also we know for $y\in \mathcal{M}_2\backslash \mathcal{M}_1$, \[(NorCone(\mu_1(\mathcal{M}_2), \mu_1(\mathcal{M}_k))_y, NorCone(\mu_1(\mathcal{M}_2), \mu_1(X))_y)\cong (\mathcal{C}_y^{k-1}(W^y_{r-1})\subset \mathbb{P}T_y(W^y_{r-1})) (\sharp).\]
		To continue using Proposition \ref{singularityafterblowup} in the next blow-up, it remains to show that $(\sharp)$ also holds over the points in the boundary of $\mu_1(\mathcal{M}_2)$ (i.e. $\mu_1(\mathcal{M}_2)\backslash (\mathcal{M}_2\backslash \mathcal{M}_1)$).
		For $y\in N_{r-1}$, we claim that  $NorCone(\mu_1(\mathcal{M}_2), \mu_1(\mathcal{M}_k))\rightarrow \mu_1(\mathcal{M}_2)$ is a trivial family over $\kappa_2^{-1}(y)$, i.e. \[NorCone(\mu_1(\mathcal{M}_2), \mu_1(\mathcal{M}_k))|_{\kappa_2^{-1}(y)} = \kappa_2^{-1}(y)\times \mathcal{C}_y^{k-1}(W^y_{r-1}).\] To prove the claim, from the proof of Corollary \ref{transversal} we know the family of balanced subspaces $\{W^{x}_{r-1}\}_{x\in \kappa_2^{-1}(y)}$ is translated by the group action of $Aut(Z_1^{y})$ on $W^{y}_{r-1}$. Also we know for $g\in Aut(Z_1^{y})$, the adjoint action $Adg$ acts on $T_y(W_{r-1}^y)$ trivially (since from the proof of Corollary \ref{transversal}, if we write $y=c_{\Pi_{r-1}}.o$, roots corresponding to the root spaces in $T_o(c_{\Pi_{r-1}}Z_{1}^y)$) and $T_o(c_{\Pi_{r-1}}W_{r-1}^y)$ are strongly orthogonal to each other). 
		Therefore the claim holds. Now by the smoothness of $\mu_1(\mathcal{M}_2)$ we know $(\sharp)$ holds over the points in the boundary.
		The rest can be done by easy induction.

	\item 
	The second statement follows from a similar argument as in (1) using Proposition \ref{singularityafterblowup} and \ref{describnor}(2). 
	\end{enumerate}
\end{proof}

\subsection{The Bia{\l}ynicki-Birula decomposition of irreducible compact Hermtian symmetric spaces}\label{BBIHSS}
In this section we give the relation between the stratification $\mathcal{M}_1\subset \mathcal{M}_2\subset \cdots \subset \mathcal{M}_r$ and the Bia{\l}ynicki-Birula decomposition of $X$.
 We first recall the Bia{\l}ynicki-Birula decomposition theorem (over $\mathbb{C}$) as follows.

\begin{theorem}[\cite{MR366940}, Theorem 4.1, 4.2, 4.3]\label{BBdecom}
	Let $X$ be a complete smooth complex manifold with an algebraic $\mathbb{C}^*$-action. Let $X^{\mathbb{C}^*}$ be the set of fixed points under the $\mathbb{C}^*$-action and $X^{\mathbb{C}^*}=\bigcup_{i=0}^r (X^{\mathbb{C}^*})_i$ be the decomposition of  $X^{\mathbb{C}^*}$ into connected components. Then there exists a unique locally closed $\mathbb{C}^*$-invariant decomposition of $X$,\[X=\bigsqcup_{i=0}^r X^+_i \left( resp. X=\bigsqcup_{i=0}^r X^-_i \right)\]
	and morphisms $\gamma^+_i: X^+_i\rightarrow (X^{\mathbb{C}^*})_i$ (resp. $\gamma^-_i: X^-_i\rightarrow (X^{\mathbb{C}^*})_i$), $0\leq i
	\leq r$, such that the following holds.
	\begin{enumerate}
		\item $X^+_i$ (resp. $X^-_i$) are smooth $\mathbb{C}^*$-invariant complex submanifolds of $X$. $(X^{\mathbb{C}^*})_i$ is a closed complex submanifold of $X^+_i$ (resp. $X^-_i$). $X^+_i\cap X^-_i=(X^{\mathbb{C}^*})_i$.
		
		\item $\gamma^+_i$ (resp. $\gamma^-_i$) is a $\mathbb{C}^*$-fibration over $(X^{\mathbb{C}^*})_i$ (i.e. each fiber is a $\mathbb{C}^*$-module) such that $\gamma^+_i|_{(X^{\mathbb{C}^*})_i}$ (resp. $\gamma^-_i|_{(X^{\mathbb{C}^*})_i}$) is the identity.
		
		\item  Let $T_x(X)^+, T_x(X)^0, T_x(X)^-$ be the weight spaces of the isotropy action on the tangent space $T_x(X)$ with positive, zero, negative weights respectively. Then for any $x\in (X^{\mathbb{C}^*})_i$, the tangent space $T_x(X^+_i)=T_x(X)^0\oplus T_x(X)^+$ (resp. $T_x(X^-_i)=T_x(X)^0\oplus T_x(X)^-$) and dimension of the fibration given by $\gamma_i^+$ (resp. ($\gamma_i^-$)) equals $\dim T_x(X)^+$ (resp. $\dim T_x(X)^-$).
	\end{enumerate}
We call $X_i^+$ (resp. $X_i^-$) a stable (resp. unstable) manifold of $X$ with respect to the $\mathbb{C}^*$-action.
\end{theorem}

We know there is a natural $\mathbb{C}^*$-action induced by the one-dimensional center of $\mathfrak{k}^{\mathbb{C}}$. Let $c$ be the central element satisfying that $[c,v]=v$ for any $v\in \mathfrak{m}^+$. Let $W\cong \mathbb{C}^n \subset X$ be the Harish-Chandra embedding given by the map $\exp: \mathfrak{m}^+ \rightarrow \exp(\mathfrak{m}^+).o$ where $o$ is the base point. Let $\omega_t=\exp(tc)$ be the $\mathbb{C}^*$-action. We can see that $o$ is fixed by $\omega_t$ and $W$ is invariant under $\omega_t$. Moreover we have

\begin{lemma}\label{BBfix}
	The points in $N_k$ are fixed by $\omega_t$.
\end{lemma}
\begin{proof}
	Recall that  \[\begin{aligned}
	N_k=&\{C\backslash W: C \ \text{is a rational curve passing through $o$ such that} \\& [T_o(C)] \ \text{is of rank} \ k\  \text{and} \ C\cap W \ \text{is an affine line}\} 
	\end{aligned}\] 
		We know $C\cap W=\{\exp(sv).o, s\in \mathbb{C}\}$ for some $c$ with $[v]\in \mathcal{C}_o^k(X)\backslash\mathcal{C}_o^{k-1}(X)$. Also $\omega_t\exp(sv).o=\exp(e^tsv).o$ and hence $\omega_s$ keeps $C\cap W$ invaraint. Thus the infinity point of $C$ is fixed by $\omega_t$.
\end{proof}

We then give the action of $\omega_t$ on the affines cells of the characteristic Hermitian symmetric subspaces contained in $\mathcal{M}_k\backslash \mathcal{M}_{k-1}$ (see Proposition \ref{fiboverinfty}). We can prove that
\begin{lemma}\label{Cstar1}
 Each fiber of $\kappa_k :\mathcal{M}_k\backslash \mathcal{M}_{k-1}\rightarrow N_{r-k+1}$ is invariant under $\omega_t$. Moreover the action is precisely the dilation with rate $e^t$ fixing the zero point in $N_{r-k+1}$.
\end{lemma}
\begin{proof}
	As in Proposition \ref{fiboverinfty} we can write a fiber of $\kappa_k$ as
	$\{c_{\Pi_{r-k+1}}\exp(v).o, v\in \mathfrak{m}_{r-k+1}^+\}$. As in Lemma \ref{inftra} we can compute
	\[\begin{aligned}
	Ad\exp(\frac{\pi}{2}x_{\alpha})c=&\exp(ad(\frac{\pi}{2}x_{\alpha}))c\\
	&=c+\frac{1}{2}\sum_{k=0}^{\infty}\frac{(\pi\sqrt{-1})^{2k+1}}{(2k+1)!}y_\alpha
	+\frac{1}{2}\sum_{k=1}^{\infty}\frac{(\pi\sqrt{-1})^{2k}}{(2k)!}h_{\alpha}\\
	&=c+\frac{1}{2}\sinh(\pi\sqrt{-1})y_{\alpha}+\frac{1}{2}(\cosh(\pi\sqrt{-1})-1)h_{\alpha}\\
	&=c-h_{\alpha}
	\end{aligned}
	\]
	Since all roots with root spaces in $\mathfrak{m}^+_{r-k+1}$ are strongly  orthogonal to $\Pi_{r-k+1}$, we know $[h_{\alpha}, v]=0$ for any $\alpha \in \Pi_{r-k+1}, v\in \mathfrak{m}^+_{r-k+1}$ and hence
	\[\begin{aligned}
	\omega_tc_{\Pi_{r-k+1}}\exp(v).o&=\exp(tc)c_{\Pi_{r-k+1}}\exp(v).o\\
	&=c_{\Pi_{r-k+1}}(Adc_{\Pi_{r-k+1}}\exp(tc))\exp(v).o\\
	&=c_{\Pi_{r-k+1}}\exp(t(c-\sum_{\alpha\in \Pi_{r-k+1}}h_{\alpha}))\exp(v).o\\
	&=c_{\Pi_{r-k+1}}\exp(Ad\exp(tc)v).o=c_{\Pi_{r-k+1}}\exp(e^tv).o
	\end{aligned}
	\]
	Thus the Lemma follows.
\end{proof}

Also we can compute the $\mathbb{C}^*$-action $\omega_t$ on the fibers of $\nu_k: \mathcal{V}_k\backslash \mathcal{V}_{k-1}$.

\begin{lemma}\label{Cstar2}
	Each  fiber of $\nu_k :\mathcal{V}_k\backslash \mathcal{V}_{k-1}\rightarrow N_k$ is invariant under $\omega_t$. Moreover the action is precisely the dilation with rate $e^{-t}$ fixing the zero point in $N_{k}$.
\end{lemma}
\begin{proof}
	As in Proposition \ref{tubefib} we can write a fiber of $\nu_k$ as
	$\{c_{\Pi_{k}}\exp(v).o, v\in \mathfrak{u}_{k}^+\}$. As in Lemma \ref{Cstar1} we have
	\[\begin{aligned}
	\omega_tc_{\Pi_{k}}\exp(v).o&=\exp(tc)c_{\Pi_{k}}\exp(v).o\\
	&=c_{\Pi_{k}}(Adc_{\Pi_{k}}\exp(tc))\exp(v).o\\
	&=c_{\Pi_{k}}\exp(t(c-\sum_{\alpha\in \Pi_{k}}h_{\alpha}))\exp(v).o\\
	\end{aligned}
	\]
	Since $\mathfrak{u}_{k}^+=\mathbb{C}\{e_{\beta}, \beta|_{\mathfrak{h}^-}=\frac{1}{2}(\alpha_{\ell}+\alpha_j), (1\leq \ell \leq j\leq k) \}$, we know $[\sum_{\alpha\in \Pi_{k}}h_{\alpha}, v]=2v$ for any $v\in \mathfrak{u}_{k}^+$. Then $	\omega_tc_{\Pi_{k}}\exp(v).o=c_{\Pi_{k}}\exp(e^{-t}v).o$
	and the Lemma follows.
\end{proof}

Then we can write down the Bia{\l}ynicki-Birula decomposition explicitly for the irreducible compact Hermtian symmetric space. We have 

\begin{proposition}
	The following setting gives the Bia{\l}ynicki-Birula decomposition for the irreducible compact Hermtian symmetric space $X$, using the notations in Theorem \ref{BBdecom}.
	Let $(X^{\mathbb{C}^*})_0=\{o\}$. Let $(X^{\mathbb{C}^*})_i=N_i (1\leq i\leq r)$. Let $X^+_0=W=X\backslash D$ be the affine cell of $X$. Let $X^+_r=\mathcal{M}_1$, $X^+_i=\mathcal{M}_{r-i+1}\backslash \mathcal{M}_{r-i} (1\leq i\leq r-1)$. Let $X^-_0=\{o\}$. Let $X^-_i=\mathcal{V}_{i}\backslash \mathcal{V}_{i-1} (1\leq i\leq r)$.
\end{proposition}
\begin{proof}
	All conditions can be checked directly from Propositions \ref{fiboverinfty}(2), \ref{tubefib} and Lemmas \ref{BBfix}, \ref{Cstar1}, \ref{Cstar2}. 
\end{proof}

$N_r$ and $\{o\}$ are called the sink and source of the $\mathbb{C}^*$-action. Sometimes we also call them the extremal fixed point components of the action. Recall that in Proposition \ref{inftypart} we know when $X$ is of tube type, $N_{i}\cong N_{r-i+1}$ and $N_r$ is a point. For simplicity we will call $N_i$ the $i$-th fixed component in the Bia{\l}ynicki-Birula decomposition.

\section{Proof of the Main theorem}\label{birtrans}
\subsection{Some lemmas}
We first give some lemmas.
\begin{lemma}\label{degreevmrt}
	For $1\leq k\leq r-1$,
	the ideal of $\mathcal{C}_o^k(X)\subset \mathbb{P}T_o(X)$ is generated by degree $k+1$ polynomials.
\end{lemma}
\begin{proof}
	This can be easily checked case by case.  When $X=G(p,q)$, we know $\mathcal{C}_o^k(G(p,q))=\{[A]\in \mathbb{P}(\mathbb{C}^p\otimes \mathbb{C}^q), A\in M(p,q,\mathbb{C}), \ \text{rank}(A)\leq k \}$. The defining equations are given by all $(k+1)\times (k+1)$ minors of the matrix. When $X=G^{II}(p,p)(p\geq 4)$, we know $\mathcal{C}_o^k(G^{II}(p,p))=\{[A]\in \mathbb{P}(\bigwedge^2\mathbb{C}^p), A\in M(p,p,\mathbb{C}) \ \text{is anti-symmetric}, \ \text{rank}(A)\leq 2k \}$. The defining equations are given by pfaffians of $(2k+2) \times (2k+2)$ anti-symmetric principal submatrices. When $X=G^{III}(p,p)(p\geq 3)$, $\mathcal{C}_o^k(X)=\{[A]\in \mathbb{P}(Sym^2\mathbb{C}^p), A\in M(p,p,\mathbb{C}) \ \text{is symmetric}, \ \text{rank}(A)\leq k \}$. The defining equations are given by  determinants of $(k+1) \times (k+1)$ anti-symmetric principal submatrices. For $Q^n$ and $\mathbb{OP}^2$, their VMRTs are $Q^{n-2}\subset \mathbb{P}^{n-1}$ and $G^{II}(5,5)\subset \mathbb{P}^{15}$, both of which are defined by quadratic polynomials. For $E_7/P_7$, its VMRT is the Cayley plane $\mathbb{OP}^2\subset \mathbb{P}^{26}$ and the secant is a  cubic hypersurface.
\end{proof}
Therefore $\varphi: \mathbb{P}^n\dashrightarrow \mathbb{P}^N$ can be written as
\[\begin{aligned}
\varphi([x_0,x_1,\cdots, x_n])&=[x^r_0,x_0^{r-1}x_1,\cdots,x_0^{r-1}x_n,x^{r-2}_0\mathcal{I}(\mathcal{C}_o^1(X)),\cdots,x_0^{r-k-1}\mathcal{I}(\mathcal{C}_o^{k}(X)),\cdots, \mathcal{I}(\mathcal{C}_o^{r-1}(X))]
\end{aligned}\]
\begin{lemma}\label{nkcoor}
The rational map 
	\[\varphi_k: [x_0, x_1,\cdots ,x_n]\in \mathbb{P}^{n}\rightarrow [0,\cdots,\mathcal{I}(\mathcal{C}_o^{k-1}(X)),\cdots,0]\in \mathbb{P}^N\]
	satisfies that $\varphi_k(\mathcal{C}_o^{k}(X))=N_k$ under the identification 	$\mathcal{C}_o^k(X)\subset \mathbb{P}T_o(X)\cong \mathbb{P}^{n-1}=\{x_0=0\}\subset \mathbb{P}^n$. Moreover the map on $\mathcal{C}_o^k(X)\backslash \mathcal{C}_o^{k-1}(X)$ is exactly the map $\theta_k$ given in Proposition \ref{fiboverinfty}.
\end{lemma}
\begin{proof}
	By the Theorem of Landsberg-Manivel (Theorem \ref{thmLM}) we know the Zariski closure of the image of $\varphi$ is isomorphic to $X$, for simplicity we also denote the closure by $X$. It is easy to observe that $\varphi$ maps the affine cell $\{x_0\neq 0 \}$ to the affine cell of $X$.
	Fix a rank $k$ tangent $v_k= [0,x_1,\cdots, x_n]\in \mathcal{C}_o^{k}(X)\backslash \mathcal{C}_o^{k-1}(X)$. We take the affine line $\{[1,tx_1,\cdots, tx_n], t\in \mathbb{C}\}$ then its image under $\varphi$ is \[\varphi([1,tx_1,\cdots, tx_n])=\{[1, t\mathbf{x}, \mathbf{F}_1(t\mathbf{x}), \cdots, \mathbf{F}_{k-1}(t\mathbf{x}),0,\cdots, 0],t\in \mathbb{C}\}\] where $\mathbf{F}_i$ denotes the vector-valued function given by the set of generators of the ideal  $\mathcal{I}(\mathcal{C}_o^{i}(X))$ and $\mathbf{x}=(x_1,\cdots x_n)$. Let $t \rightarrow \infty$. Note that $\mathbf{F}_i(t\mathbf{x})=t^{i+1}\mathbf{F}(\mathbf{x})$, we have 
	\[\begin{aligned}
	&[1, t\mathbf{x}, \mathbf{F}_1(t\mathbf{x}), \cdots, \mathbf{F}_{k-1}(t\mathbf{x}),0,\cdots, 0]=\\&[1, t\mathbf{x}, t^2\mathbf{F}_1(\mathbf{x}), \cdots, t^k\mathbf{F}_{k-1}(\mathbf{x}),0,\cdots, 0]\rightarrow [0,\cdots,\mathbf{F}_{k-1}(\mathbf{x}),\cdots,0]
	\end{aligned}\]
	Then by definition this is precisely a infinity point in $N_k$.
	\end{proof}

Let $\mathbb{P}_{k}=span<N_{k}>$ be the linear subspace contained in $\mathbb{P}^N$. Note that if $f\in \mathcal{I}(\mathcal{C}_o^{k-1}(X))$ is a generator, it can not be identically vanishing on $\mathcal{C}_o^{k}(X)$ by the degree restriction in Lemma \ref{degreevmrt}. Then we know $\dim(\mathbb{P}_{k})+1$ equals to the number of generators in $\mathcal{I}(\mathcal{C}_o^{k-1}(X))$ and moreover $\mathbb{P}_{k}$ can be identified with projectivization of a $K^{\mathbb{C}}$-module.

\begin{lemma}\label{inversionmatrix}
	When $X$ is of tube type with rank at least $3$, $\varphi_{r-1}|_{x_0=0}$ is a birational map onto $\mathbb{P}_{r-1}$.
\end{lemma}
\begin{proof}
	We can check case by case. When $X=G(m,m)$, we know $\mathcal{C}_o^{r-2}(G(m,m))=\{[A]\in \mathbb{P}(\mathbb{C}^m\otimes \mathbb{C}^m), A\in M(m,m,\mathbb{C}), \ \text{rank}(A)\leq r-2 \}$. The defining equations are given by all $(r-1)\times (r-1)$ minors of the matrix. Then we know $\dim(\mathbb{P}_{r-1})=n-1=m^2-1$ and the rational map $\varphi_{r-1}|_{\{x_0=0\}}$ is to take the inverse of (the projectivization) of  $m\times m$ matrices if we identify $\{x_0=0\}=\mathbb{P}^{n-1}$ with the projectivized tangent space of $X$ at $o$. Similarly we know in the orthogonal Grassmannian $G^{II}(m,m)$ ($m$ is even) and Lagrangian Grassmannian  $G^{III}(m,m)$ case the rational map $\varphi_{r-1}|_{\{x_0=0\}}$ is to take the inverse of (the projectivization) of  $m\times m$ anti-symmetric and symmetric matrices respectively. When $X$ is $E_7/P_7$, also by a similar argument we know the rational map $\varphi_{r-1}|_{\{x_0=0\}}$ is to take the inverse of (the projectivization) of $3\times 3$ matrices in the Jordan algebra $J_3(\mathbb{O})$.
\end{proof}

\subsection{Stratifications as linear sections of $X$}
For later discussion, we fix some notations here. Let $\mathbb{P}_{I_j}=span<\mathbb{P}_{r},\cdots,\mathbb{P}_{r-j+1}>$ for $1\leq j\leq r$ and $\mathbb{P}_{I_{r+1}}=\mathbb{P}^N$. For simplicity we use $\mathbb{P}_0$ to denote tha point set $\{[1,0,\cdots,0]\}\subset \mathbb{P}^N$. Then let $\mathbb{P}_{J_j}=span<\mathbb{P}_{0},\mathbb{P}_1,\cdots,\mathbb{P}_{j}>$ for $0\leq j\leq r$.
Then we describe the stratifications $\mathcal{M}_1\subset \mathcal{M}_2\subset \cdots \subset \mathcal{M}_{r-1}\subset \mathcal{M}_r=D$ and $\mathcal{V}_1\subset \mathcal{V}_2 \subset \cdots \subset \mathcal{V}_{r-1}$ as follows.

\begin{lemma}\label{linearsec}
	$\mathcal{M}_k=\mathbb{P}_{I_k}\cap X (1\leq k\leq r)$ is a linear section of $X$. In particular, $D=\mathcal{M}_r$ is a hyperplane section of $X$.
\end{lemma}
\begin{proof}
	We first prove $\mathcal{M}_k\subset \mathbb{P}_{I_k}\cap X$ by induction. $\mathcal{M}_1\subset \mathbb{P}_{I_1}\cap X$ is clear. Now suppose that $\mathcal{M}_{j}\subset \mathbb{P}_{I_j}\cap X$. From Proposition \ref{fiboverinfty}(2) we know $N_{r-j}\subset \mathcal{M}_{j+1}$ and $\mathcal{M}_{j+1}\backslash \mathcal{M}_{j}$ is a homogeneous bundle over $N_{r-j}$ whose fiber is the affine cell of a Hermitian characteristic symmetric subspace of rank $j$. Fix a point $x \in N_{r-j}$ and denote the corresponding Hermitian characteristic symmetric subspace by $Z_{j}^x$. The zero point in the Harish-Chandra coordinate of $Z_j^x$ is $x$. Let $\mathcal{V}^{Z^x_j}_s$ be the $s$-th locus of chains of minimal rational curves in $Z^x_j$ starting from $x$, i.e. we inductively define $\mathcal{V}^{Z^x_j}_{0}=\{x\}$ and
   \[\mathcal{V}^{Z^x_j}_{s}=\overline{\bigcup\{\text{mrcs in}\  Z^x_j \ \text{passing} \ \mathcal{V}^{Z^x_j}_{s-1}\}}.\] 
	Let $D^{Z^x_j}$ be the compactifying divisor of $Z^x_j$ and it is contained in $\mathcal{M}_{j}\subset \mathbb{P}_{I_j}\cap X$ by inductive  assumption. Also we know any minimal rational curve in $Z^x_j$ passing $\mathcal{V}^{Z^x_j}_{s-1}$ is a line in $\mathbb{P}^N$ connecting $\mathcal{V}^{Z^x_j}_{s-1}$ and $D^{Z^x_j}$. Then by simple induction on $s$ we know $\mathcal{V}^{Z^x_j}_{s}\subset span<x,\mathcal{M}_j>\subset \mathbb{P}_{I_{j+1}}$ and moreover in the same induction we can obtain that 
	the affine cell is disjoint from $\mathbb{P}_{I_j}$. By varying $x\in N_{r-j}$  we know $\mathcal{M}_{j+1}\subset \mathbb{P}_{I_{j+1}}\cap X$. Therefore $\mathcal{M}_{k}\subset  \mathbb{P}_{I_{k}}\cap X$ for all $1\leq k\leq r$. Also from the proof we can see that for any point $y\in \mathcal{M}_{\ell+1}\backslash \mathcal{M}_{\ell} (\ell\geq k)$, it can be projected to a point in $N_{r-\ell}$ (namely, the zero point of the corresponding affine cell given by Harish-Chandra embedding) from $\mathbb{P}_{I_{\ell}}$. Then $y\notin \mathbb{P}_{I_{k}}\cap X$ and hence $\mathcal{M}_{k}= \mathbb{P}_{I_{k}}\cap X$ for all $1\leq k\leq r$.
\end{proof}

Similarly we can prove that 
\begin{lemma}\label{linearsec2}
	$\mathcal{V}_k=\mathbb{P}_{J_k}\cap X (0\leq k\leq r-1)$ is a linear section of $X$. 
\end{lemma}
\begin{proof}
	The proof is similar to the proof of Lemma \ref{linearsec}. We first prove $\mathcal{V}_k\subset \mathbb{P}_{J_k}\cap X$ by induction.
	We know $\mathcal{V}_0=\{o\}=\{[1,0\cdots,0]\}$. Now suppose that $\mathcal{V}_j\subset \mathbb{P}_{J_j}\cap X$. From Proposition \ref{tubefib} we know $N_{j+1}\subset \mathcal{V}_{j+1}$ and $\mathcal{V}_{j+1}\backslash \mathcal{V}_{j}$ is a homogeneous bundle over $N_{j+1}$ whose fiber is the affine cell of a balanced subspace of rank $j+1$. Fix a point $x \in N_{j+1}$ and denote the corresponding balanced subspace by $W_{j+1}^x$. The zero point of $W_{j+1}^x$ in Harish-Chandra coordinate is $x$. Let $\mathcal{V}^{W^x_{j+1}}_s$ be the $s$-th locus of chains of minimal rational curves in $W^x_{j+1}$ starting from $x$, i.e. we inductively define $\mathcal{V}^{W^x_{j+1}}_{0}=\{x\}$ and
	\[\mathcal{V}^{W^x_{j+1}}_{s}=\overline{\bigcup\{\text{mrcs in}\  W^x_{j+1} \ \text{passing} \ \mathcal{V}^{W^x_{j+1}}_{s-1}\}}.\] 
	Let $D^{W^x_{j+1}}$ the compactitying divisor of $W^x_{j+1}$ with respect to this affine cell (as a fiber in the homogeneous bundle $\mathcal{V}_{j+1}\backslash \mathcal{V}_{j} \rightarrow N_{j+1}$). We know it is contained in $\mathcal{V}_j\subset \mathbb{P}_{J_j}\cap X$ by inductive  assumption. Also we know any minimal rational curve in $W^x_{j+1}$ passing $\mathcal{V}^{W^x_{j+1}}_{s-1}$ is a line in $\mathbb{P}^N$ connecting $\mathcal{V}^{W^x_{j+1}}_{s-1}$ and $D^{W^x_{j+1}}$. Then by simple induction on $s$ we know $\mathcal{V}^{W^x_{j+1}}_{s}\subset span<x,\mathcal{V}_j>\subset \mathbb{P}_{J_{j+1}}$ and moreover in the same induction we can obtain that 
	the affine cell is disjoint from $\mathbb{P}_{J_j}$. By varying $x\in N_{r-j}$  we know $\mathcal{V}_{j+1}\subset \mathbb{P}_{J_{j+1}}\cap X$. Therefore $\mathcal{V}_{k}\subset  \mathbb{P}_{J_{k}}\cap X$ for all $0\leq k\leq r-1$. Also from the proof we can see that for any point $y\in \mathcal{V}_{\ell+1}\backslash \mathcal{V}_{\ell} (\ell\geq k)$, it can be projected to a point in $N_{\ell+1}$ (namely, the zero point of the corresponding affine cell given by Harish-Chandra embedding) from $\mathbb{P}_{J_{\ell}}$. Then $y\notin \mathbb{P}_{J_{k}}\cap X$ and hence $\mathcal{V}_{k}= \mathbb{P}_{J_{k}}\cap X$ for all $0\leq k\leq r-1$.
\end{proof}


\begin{remark}
Using the notations in Lemma \ref{nkcoor} we describe the $\mathbb{C}^*$-orbit in $X$ whose closure passes $o=[1,0,\cdots, 0]\in \mathbb{P}^N$ as follows. For a fixed rank $k$ tangent $v_k=[0,x_1,\cdots, x_n]\in \mathcal{C}_o^k(X)\backslash \mathcal{C}_o^{k-1}(X)$, the corresponding $\mathbb{C}^*$-orbit is given by 
\[
\{[1, t\mathbf{x}, t^2\mathbf{F}_1(\mathbf{x}), \cdots, t^k\mathbf{F}_{k-1}(\mathbf{x}),0,\cdots, 0], t\in \mathbb{C}\backslash \{0\}\}
\]
whose closure is a rational curve of degree $k$ in $\mathbb{P}^N$. If we consider the $\mathbb{C}^*$-action on $\mathbb{P}^N$ given by $t.[z, \mathbf{z}_0,\mathbf{z}_1\cdots, \mathbf{z}_{r-1}]=[z, t\mathbf{z}_0\cdots,  t^r\mathbf{z}_{r-1}]$ where $[\mathbf{z}_{k-1}]\in \mathbb{P}^{p_k-1}$
we can see the embedding of $X$ in $\mathbb{P}^N$ is $\mathbb{C}^*$-equivariant with respect to the actions on $X$ and $\mathbb{P}^N$.

It is easy to write down the Bia{\l}ynicki-Birula decomposition for the $\mathbb{C}^{*}$-action on $\mathbb{P}^N$. We can see the fixed point components are 
$(\mathbb{P}^N)^{\mathbb{C}^*}_i=\mathbb{P}_i(0\leq i\leq r)$ and $(\mathbb{P}^N)^+_0=\mathbb{P}^N\backslash \mathbb{P}_{I_r}, (\mathbb{P}^N)^+_i=\mathbb{P}_{I_{r-i+1}}\backslash \mathbb{P}_{I_{r-i}} (1\leq i\leq r-1), (\mathbb{P}^N)^+_r=\mathbb{P}_{r}$; $(\mathbb{P}^N)^-_0=\mathbb{P}_0, (\mathbb{P}^N)^-_i=\mathbb{P}_{J_i}\backslash\mathbb{P}_{J_{i-1}} (1\leq i\leq r)$. This explains Lemma \ref{linearsec} and Lemma \ref{linearsec2}.
\end{remark}

\subsection{Proof of the Main Theorem}
Let $\text{Bl}'_kX$ denote the $k$-th successive blow-up of $X$ along (strict transforms) of $\mathcal{V}_0=\{o\}\subset \mathcal{V}_1\subset  \cdots \subset \mathcal{V}_{k-1} (1\leq k\leq r-1)$. Let $\mathcal{T}_1\cong \mathbb{P}T_oX$ be the exceptional divisor in the first blow-up and $\mathcal{T}_k$ be the strict transform of $\mathcal{T}_1$ in the next successive blow-ups. Let $\text{Bl}'_k\mathbb{P}^N$ denote the $k$-th successive blow-up of $X$ along (strict transforms) of $\mathbb{P}_{J_0}\subset \mathbb{P}_{J_1}\subset \cdots \subset \mathbb{P}_{J_{k-1}}$. From Lemma \ref{linearsec2} we know $\text{Bl}'_kX$ can be naturally embedded into $\text{Bl}'_k\mathbb{P}^N$. Then we have $\mathcal{T}_{k}\subset \text{Bl}'_{k}X \subset \text{Bl}'_k\mathbb{P}^N$. Let $p'_k:\mathbb{P}^N \dashrightarrow \mathbb{P}_{I_{r-k+1}}$ be the projection map. We know the $k$-th successive blow-up on $\mathbb{P}^N$ resolves $p'_k$. The lifted morphism is denoted by $\tilde{p'_k}: \text{Bl}'_k\mathbb{P}^N \rightarrow \mathbb{P}_{I_{r-k+1}}=span<\mathbb{P}_r, \cdots, \mathbb{P}_{k}>$. Let $\iota_k: \text{Bl}'_{k}\mathbb{P}^N \rightarrow \text{Bl}'_{1}\mathbb{P}^N$ be the (successive) blow-up morphism and the lifted rational map induced by the first blow-up is denoted by $p'_{k,1}: \text{Bl}'_{1}\mathbb{P}^N \dashrightarrow \mathbb{P}_{I_{r-k+1}}$. Using commutative diagram we have 
	\[	
\begin{tikzcd}
\text{Bl}'_{k}\mathbb{P}^{N}\arrow{dr}{
	\tilde{p'_k}} \arrow{r}{\iota_k} & \text{Bl}'_{1}\mathbb{P}^{N}\arrow[dashed]{d}{p'_{k,1}}\arrow{r} & \mathbb{P}^{N}\arrow[dashed]{dl}{p'_k}\\	
& \mathbb{P}_{I_{r-k+1}}
\end{tikzcd}
\]

We prove the following proposition.
\begin{proposition}\label{imageofex}
When $X$ is of tube type with rank at least $3$,  $p'_{r-1,1}|_{\mathcal{T}_1}$ is birational onto $\mathbb{P}_{r-1}$. 
\end{proposition}
\begin{proof}
We know $\text{Bl}'_{k}\mathbb{P}^N$ is a subvariety in $ \mathcal{G}_k:=\mathbb{P}^N\times\mathbb{P}_{I_r}\times \cdots \times \mathbb{P}_{I_{r-k+1}}$ and $\tilde{p'_k}$ is the restriction of the projection from $\mathcal{G}_k$ to $\mathbb{P}_{I_{r-k+1}}$. In order to give the restriction of $\tilde{p'_k}$ on $\mathcal{T}_k\subset \text{Bl}'_{k}X\subset \text{Bl}'_{k}\mathbb{P}^N$, we consider the strict transform of the following curve where $\mathbf{x}=(x_1,\cdots, x_n)$
\[
C_j:=\{[1, t\mathbf{x}, t^2\mathbf{F}_1(\mathbf{x}), \cdots, t^j\mathbf{F}_{j-1}(\mathbf{x}),0,\cdots, 0], t\in \mathbb{C}\}.
\] 
Let $C^k_j$ be the strict transform of $C_j$ in $\text{Bl}'_{k}X$ ($k\leq j$).
Note that through the embedding $ \text{Bl}'_{k}X\subset \text{Bl}'_{k}\mathbb{P}^N$, $C^k_j$ can be identified with the strict transform of $C_j$
in the blow-ups on $\mathbb{P}^N$. We know in $\mathcal{G}_1=\mathbb{P}^N\times \mathbb{P}_{I_r}$, we have 
\[C_j^1\backslash o=\{[1, t\mathbf{x}, t^2\mathbf{F}_1(\mathbf{x}), \cdots, t^j\mathbf{F}_{j-1}(\mathbf{x}),0,\cdots, 0]\times[t\mathbf{x}, t^2\mathbf{F}_1(\mathbf{x}), \cdots, t^j\mathbf{F}_{j-1}(\mathbf{x}),0,\cdots, 0] , t\in \mathbb{C}\backslash\{0\}\}\]
Letting $t\rightarrow 0$ we can obtain that
\[C_j^1\cap \mathcal{T}_1=\{[1,0,\cdots, 0]\times [\mathbf{x},0,\cdots, 0]\}\]
Similarly we have  
\[C_j^k\cap \mathcal{T}_k=\{[1,0,\cdots, 0]\times [\mathbf{x},0,\cdots, 0]\times [\mathbf{F}_1(\mathbf{x}),0,\cdots, 0]\times \cdots \times [\mathbf{F}_{k-1}(\mathbf{x}),0,\cdots,0]\}\]
Then $\tilde{p'_k}(C_j^k\cap \mathcal{T}_k)=\{[\mathbf{F}_{k-1}(\mathbf{x}),0,\cdots,0]\}\subset \mathbb{P}_k\subset  \mathbb{P}_{I_{r-k+1}}$. In particular when $k=r-1$, $\tilde{p'_{k}}(\mathcal{T}_{k})=\mathbb{P}_{r-1}$ by Lemma \ref{inversionmatrix}. 
\end{proof}

Now we can prove 
\begin{proposition}\label{contraction}
	The successive blow-ups on $X$ resolves the rational map $\psi$ and gives a morphism $\psi' : \text{Bl}_{r-1}X \rightarrow \mathbb{P}^n$. In other words, we have the following commutative diagram.
	\[	
	\begin{tikzcd}
	\text{Bl}_{r-1}X\arrow{dr}{
		\psi'} \arrow{r} & X\arrow[dashed]{d}{\psi}\\	
	& \mathbb{P}^n
	\end{tikzcd}
	\]
	Moreover the morphism $\psi'$ maps the divisor $\mu^{(r-j-1)}_{r-1}(E^X_j) \subset \text{Bl}_{r-1}X$  to $\mathcal{C}_o^{r-j+1}(X)\subset \{x_0=0\}\subset \mathbb{P}^n$ for $1\leq j\leq r-1$ and it maps the divisor $\mu_{r-1}(D)$ to $\mathcal{C}_o^1(X)\subset \{x_0=0\}\subset \mathbb{P}^n$.
\end{proposition}

\begin{proof}
	We can see that the inverse of $\varphi$, denoted by $\psi$, is exactly restriction of the projection of $\mathbb{P}^N$ from $\mathbb{P}_{I_{r-1}}$ to $\mathbb{P}_{J_1}$. Then from Lemma \ref{linearsec} indeterminacy of $\psi$ is exactly $\mathcal{M}_{r-1}$. Let $\text{Bl}_j\mathbb{P}^N (1
	\leq j\leq r-1)$ be the successive blow-up along (strict transforms of) $\mathbb{P}_{I_1}\subset\cdots\subset \mathbb{P}_{I_{j}}$. From Lemma \ref{linearsec} we know $\text{Bl}_jX$ can be naturally embedded in $\text{Bl}_j\mathbb{P}^N$ (Warning: $\text{Bl}_jX$ is not equal to the successive strict transforms of $X$ in the blow-ups on $\mathbb{P}^N$). Moverover, let $p$ be the projection of $\mathbb{P}^N$ from $\mathbb{P}_{I_{r-1}}$. Then the blow-ups on $\mathbb{P}^N$ resolve $p$ and hence $\psi$. We denote the the morphism $\text{Bl}_{r-1}\mathbb{P}^N\rightarrow \mathbb{P}_{J_1}$ by $\tilde{p}$ and the morphism $\text{Bl}_{r-1}X \rightarrow \mathbb{P}_{J_1}$ by $\tilde{\psi}$.
	We know after $(r-1)$-times successive blow-ups, total transform of the compactifying divisor $D$ is 
	\[\bigcup_{j=1}^{r-1}\mu^{(r-j-1)}_{r-1}(E^X_j)\cup \mu_{r-1}(D)\]
	where we recall that $\mu^{(r-j-1)}_{r-1}(E^X_j)$ is the strict transform of the $j$-th exceptional divisor in the next $(r-j-1)$-times blow-ups and $\mu_{r-1}(D)$ is the strict transform of $D$ in the $(r-1)$-times blow-ups. Fix a point $x\in N_{r-j+1}$, from Proposition \ref{tubefib} we know it is corresponding to a balanced subspace $W^x_{r-j+1}$ of rank $r-j+1$. When we do the $j$-th blow-up, we know the $j$-th exceptional divisor $E_j^X$ 
	is a $\mathbb{P}^{d_j}$-bundle over $\mu_{j-1}(\mathcal{M}_j)$ as the latter is smooth from Proposition \ref{stofcha}. Also we know the strict transform of $\mathcal{V}^{W^x_{r-j+1}}_{s}$ is  $\text{Bl}_{x}\mathcal{V}^{W^x_{r-j+1}}_{s}$ and moreover \[\text{Bl}_{x}\mathcal{V}^{W^x_{r-j+1}}_{s}\cap E^X_{j}=TanCone(\mathcal{V}^{W^x_{r-j+1}}_{s},x)=\mathcal{C}_x^s(W^x_{r-j+1}) \]
	In particular when $s=r-j+1$,
	\[\text{Bl}_{x}W^x_{r-j+1}\cap E^X_{j}=\mathbb{P}T_x(W^x_{r-j+1}).\]
	Thus the fiber in $E^X_j$ over $x$, denoted by $E^X_{j,x}$ is identified with $\mathbb{P}T_x(W^x_{r-j+1})$. 

	 When we do the blow-up along $\mu_{j-1}(\mathcal{M}_{j})$, we use $\tilde{\psi}_j$ to denote the morphsim $\text{Bl}_jX\rightarrow\mathbb{P}_{J_{r-j}}$ which is just restriction of the natural morphism $\tilde{p}_j: \text{Bl}_j\mathbb{P}^N\rightarrow \mathbb{P}_{J_{r-j}}$ resolving the projection $p_j: \mathbb{P}^N \dashrightarrow \mathbb{P}_{J_{r-j}}$.
	 
	  Let $\mathbb{P}^x=span<W^x_{r-j+1}> \subset \mathbb{P}^N$. We know $W^x_{r-j+1}=\mathbb{P}^x\cap X$ and $W^x_{r-j+1} \subset \mathbb{P}^x$ is also the minimal equivariant embedding of $W^x_{r-j+1}.$ Recall that $N^{W^x_{r-j+1}}_k=N_{k}\cap W^x_{r-j+1}$ is the $k$-th locus of infinity points (or the $k$-th fixed point component in the  Bia{\l}ynicki-Birula decomposition of $W^x_{r-j+1}$). Then we let $\mathbb{P}^x_k=span<N^{W^x_{r-j+1}}_k>=\mathbb{P}^x\cap \mathbb{P}_k$. Now from Proposition \ref{imageofex} we know the image of $\mu^{(r-j-1)}_{r-1}(E^X_{j,x})=\mu^{(r-j-1)}_{r-1}(\mathbb{P}T_x(W^x_{r-j+1}))$ under the morphism $\tilde{\psi}_{r-1}(=\psi'): \text{Bl}_{r-1}X \rightarrow \mathbb{P}_{J_1}$ is $\mathbb{P}_1^x\subset \mathbb{P}_1$. (Note that here in $W_{r-j+1}^{x}$ we choose the affine cell such that $x$ is the 'zero point' instead of $o$, so by Proposition \ref{imageofex} the image is $\mathbb{P}_1^x$ instead of $\mathbb{P}^x_{r-j}$.) When we vary $x\in N_{r-j+1}=K^{\mathbb{C}}/Q_{r-j+1}$, we obtain a family of balanced subspaces $\{W^x_{r-j+1}\}_{x\in N_{r-j+1}}$ of rank $r-j+1$ passing $o$ and they are transitive by the action of $K^{\mathbb{C}}$. Also we know everything including $\mathbb{P}^x_k$ is transformed by the action of $K^{\mathbb{C}}$. In particular for $\mathbb{P}^x_1\cong \mathbb{P}T_o(W^x_{r-j+1})$, we know the $K^{\mathbb{C}}$-action on $\mathbb{P}_1$ is the same as the isotropy action of $K^{\mathbb{C}}$ on $\mathbb{P}T_{o}(X)$. Then by Proposition \ref{fiboverinfty}(1), we have \[
	  \begin{aligned}
	  \psi'(\mu^{(r-j-1)}_{r-1}(\bigcup_{x\in N_{r-j+1}}E^X_{j,x}))&=\tilde{\psi}_{r-1}(\mu^{(r-j-1)}_{r-1}(\bigcup_{x\in N_{r-j+1}}E^X_{j,x}))=K^\mathbb{C}.\mathbb{P}T_o(W_{r-j+1}^x)
	  \\&=\mathcal{C}_o^{r-j+1}(X)\subset \{x_0=0\}\subset \mathbb{P}^n
	  \end{aligned}
	  \]
	  We then consider all fibers in $E_j\rightarrow \mu_{j-1}(\mathcal{M}_j)$. We claim that for $y\in \mathcal{M}_j\backslash \mathcal{M}_{j-1}$, \[\psi'(\mu^{(r-j-1)}_{r-1}(E^X_{j,y}))=\psi'(\mu^{(r-j-1)}_{r-1}(E^X_{j,\kappa_j(y)}))\] where $\kappa_j:\mathcal{M}_j\backslash \mathcal{M}_{j-1}\rightarrow N_{r-j+1}$ is the fibration given in Proposition \ref{fiboverinfty}(2). From Proposition \ref{fiboverinfty}(2), we know $y$ is in the affine cell of a Hermitian  characteristic symmetric subspace $Z^{\kappa_j(y)}_{j-1}$ with zero point $\kappa_j(y)\in N_{r-j+1}$ and rank $j-1$. Then there is a balanced subspace $W^{y}_{r-j+1}$ passing $y$ which is translated by a group action $g\in Aut(Z^{\kappa_j(y)}_{j-1})\subset G$ on $W^{\kappa_j(y)}_{r-j+1}$. Since $g$ acts trivially on $\mathbb{P}_{J_{r-j}}$ and $\mathbb{P}^{\kappa_j(y)}=span<W^{\kappa_j(y)}_{r-j+1}>= \kappa_j(y)+\mathbb{P}_{J_{r-j}}\cap \mathbb{P}^{\kappa_j(y)}$, we have $\mathbb{P}^{y}=span<W^y_{r-j+1}>=g.\mathbb{P}^{\kappa_j(y)}= g.\kappa_j(y)+\mathbb{P}_{J_{r-j}}\cap \mathbb{P}^{\kappa_j(y)}=y+\mathbb{P}_{J_{r-j}}\cap \mathbb{P}^{\kappa_j(y)}$. Moreover we know the image of strict transforms of $E^X_{j,y}=\mathbb{P}T_{y}(W^y_{r-j+1})$ and $E^X_{j,\kappa_j(y)}=\mathbb{P}T_{\kappa_j(y)}(W^{\kappa_j(y)}_{r-j+1})$ under the morphism $\psi'$ are the same.
	  In conculsion we have
	  \[\psi'(\mu^{(r-j-1)}_{r-1}(\bigcup_{x\in \mathcal{M}_j\backslash \mathcal{M}_{j-1}}E^X_{j,x}))=\psi'(\mu^{(r-j-1)}_{r-1}(\bigcup_{x\in N_{r-j+1}}E^X_{j,x}))=\mathcal{C}_o^{r-j+1}(X)\subset \{x_0=0\}\subset \mathbb{P}^n.\]
	  Since $\mu_{j-1}(\mathcal{M}_{j})$ is the closure of $\mathcal{M}_{j}\backslash \mathcal{M}_{j-1}$, we finish the proof of the proposition.
\end{proof}

Now we are ready to finish the proof of the Main Theorem.

\begin{proof}
	From Proposition \ref{contraction} and the universal property of blow-up we 
	know there is a surjective morphism $\psi^1$ such that the following commutative diagram holds
	\[	
	\begin{tikzcd}
	\text{Bl}_{r-1}X\arrow{dr}{
		\psi'} \arrow{d}{\psi^1} \arrow{r} & X\arrow[dashed]{d}{\psi}\\	
	\text{Bl}_1\mathbb{P}^n \arrow{r} & \mathbb{P}^n 
	\end{tikzcd}
	\]
	Then we can see that by the morphism $\psi^1$, the divisor $\tau^{(r-k-1)}_{r-1}(E^X_k) \subset \text{Bl}_{r-1}X$  is mapped to $\mu_1(\mathcal{C}
	_o^{r-k+1}(X))$ for $1\leq k\leq r-1$ and the divisor $\mu_{r-1}(D)$ is mapped to $E^{\mathbb{P}^n}_1$. If we continue using the universal property of blow-ups, we obtain the following commutative diagram
	\[	
	\begin{tikzcd}
	&&&	\text{Bl}_{r-1}X\arrow{llld}{\psi^{r-1}}\arrow{lld}{\psi^k}\arrow{ld}{\psi^2}\arrow{dr}{
		\psi'} \arrow{d}{\psi^1} \arrow{r} & X\arrow[dashed]{d}{\psi}\\	
	\text{Bl}_{r-1}\mathbb{P}^n\arrow{r}&\cdots\text{Bl}_{k}\mathbb{P}^n\arrow{r}&\cdots \text{Bl}_2\mathbb{P}^n \arrow{r}&	\text{Bl}_1\mathbb{P}^n \arrow{r} & \mathbb{P}^n 
	\end{tikzcd}
	\]
	where all $\psi^k(1\leq k\leq r-1)$ are surjective morphisms. Write $\tilde{\psi}=\varphi^{r-1}$.
	Then by the morphism $\tilde{\psi}$, the divisor $\mu^{(r-k-1)}_{r-1}(E^X_k) \subset \text{Bl}_{r-1}X$  is mapped to $\tau^{(k-2)}_{r-1}(E_{r-k+1}^{\mathbb{P}^n})$ for $2\leq k\leq r$, the divisor $\mu^{(r-2)}_{r-1}(E^X_1) \subset \text{Bl}_{r-1}X$ is mapped to $\tau_{r-1}(\mathbb{P}^{n-1})$
	and the divisor $\mu_{r-1}(D)$ is mapped to $\tau^{(r-2)}_{r-1}(E^{\mathbb{P}^n}_1)$. 
	Note that in the blow-ups all centers are smooth, we know $\text{Bl}_{r-1}\mathbb{P}^n$ and $\text{Bl}_{r-1}X$ are smooth. Since $\tilde{\psi}$ is birational morphism between two smooth varieties with the same Picard number, $\tilde{\varphi}$ must be an isomorphism and then the Main Theorem holds.
\end{proof}

\bibliographystyle{alpha}
\bibliography{research} 	
\end{document}